\documentclass[oneside]{amsart}
\usepackage{amsmath}
\usepackage{amsfonts}
\usepackage{amsthm}
\usepackage{hyperref}
\usepackage{amscd}
\usepackage{color}
\usepackage{enumerate}
\usepackage{xypic}
\usepackage{tikz}
\usepackage{tikz-cd}
\usepackage{blkarray}
\usepackage[backgroundcolor=white]{todonotes}

\newcommand{\balpha}{\boldsymbol{\alpha}}

\newcommand{\cd}[1]{\mathbf{u}(#1)}
\newcommand{\cdj}[1]{\mathbf{u}_j(#1)}
\newcommand{\mcd}[1]{-\mathbf{u}(#1)}
\newcommand{\G}{\mathcal{G}}
\newcommand{\bF}{\mathbf{F}}
\newcommand{\F}{\mathcal{F}}
\newcommand{\bA}{\mathbb{A}}
\newcommand{\A}{\mathcal{A}}

\newcommand{\ZZ}{\mathbb{Z}}
\newcommand{\KK}{\mathbb{K}}
\newcommand{\RR}{\mathbb{R}}
\newcommand{\PP}{\mathbb{P}}

\newcommand{\NN}{\mathbb{N}}
\newcommand{\N}{\mathcal{N}}
\newcommand{\E}{\mathcal{E}}
\newcommand{\cS}{\mathcal{S}}
\newcommand{\CO}{\mathcal{O}}
\newcommand{\mcR}{\mathcal{R}}
\newcommand{\mcL}{\mathcal{L}}
\newcommand{\mcH}{\mathcal{H}}

\newcommand{\mcZ}{\mathcal{Z}}
\newcommand{\pr}{\mathbf{p}}
\newcommand{\mfm}{\mathfrak{m}}
\newcommand{\V}{\mathcal{V}}

\DeclareMathOperator{\Pic}{Pic}

\DeclareMathOperator{\Bl}{Bl}
\DeclareMathOperator{\conv}{conv}
\DeclareMathOperator{\Sym}{Sym}
\DeclareMathOperator{\Proj}{Proj}
\DeclareMathOperator{\spec}{Spec}
\DeclareMathOperator{\Hom}{Hom}
\DeclareMathOperator{\pos}{pos}
\DeclareMathOperator{\codim}{codim}
\DeclareMathOperator{\Div}{div}

\DeclareMathOperator{\Gr}{Gr}

\newtheorem{thm}{Theorem}[section]
\newtheorem{prop}[thm]{Proposition}
\newtheorem{lemma}[thm]{Lemma}
\newtheorem{cor}[thm]{Corollary}

\theoremstyle{definition}
\newtheorem{warning}[thm]{Warning}

\newtheorem{defn}[thm]{Definition}
\newtheorem{rem}[thm]{Remark}

\newtheorem{ex}[thm]{Example}

\title{Fano Schemes of Complete Intersections in Toric Varieties}

\author{Nathan Ilten}
\address{Department of Mathematics, Simon Fraser University,
8888 University Drive, Burnaby BC V5A1S6, Canada}
\email{nilten@sfu.ca}

\author{Tyler L. Kelly}
\address{School of Mathematics, University of Birmingham, Edgbaston, Birmingham B15 2TT, United Kingdom}
\email{t.kelly.1@bham.ac.uk}
\begin{document}

\begin{abstract}
We study Fano schemes $\bF_k(X)$ for complete intersections $X$ in a projective toric variety $Y\subset \PP^n$. Our strategy is to decompose $\bF_k(X)$ into closed subschemes based on the irreducible decomposition of $\bF_k(Y)$ as studied by Ilten and Zotine. We define the ``expected dimension'' for these subschemes, which always gives a lower bound on the actual dimension. Under additional assumptions, we show that  these subschemes are non-empty and smooth of the expected dimension. Using tools from intersection theory, we can apply these results to count the number of linear subspaces in $X$ when the expected dimension of $\bF_k(X)$ is zero. 
\end{abstract}
\maketitle

\section{Introduction}
\subsection{Background}
Let $X\subset \PP^n$ be a projective variety. The $k$th Fano scheme of $X$ is the fine moduli space $\bF_k(X)$ parametrizing $k$-dimensional linear subspaces of $\PP^n$ contained in $X$. The study of $\bF_k(X)$ is classical, going back to Cayley and Salmon \cite{cayley}, who showed that a smooth cubic surface contains exactly $27$ lines, and Schubert \cite{schubert}, who showed that a general quintic threefold has $2875$ lines on it. In the second half of the twentieth century, as a striking application, Clemens and Griffiths used the Fano scheme of lines to prove the irrationality of cubic threefolds \cite{clemensgriffiths}. Fano schemes were then studied extensively when $X$ is a \emph{general} hypersurface or complete intersection, with key results determining dimension, smoothness, and connectedness of $\bF_k(X)$ \cite{altman:77a, barth:81a,langer:97a,debarre:98a}. Some of these results have been partially extended to weaken the assumption of genericity of $X$ to include smooth hypersurfaces, see \cite{harris:98a,roya1,roya2,riedl}.

More recently, Fano schemes $\bF_k(X)$ have been studied for certain \emph{special} varieties $X$, for example, when $X$ is cut out by the $m\times m$ minors  of a generic matrix \cite{chan}. An understanding of the Fano schemes for special varieties often leads to interesting applications. Results on the Fano schemes of the hypersurface cut out by a sum of products of independent linear forms lead to non-trivial lower bounds on product and tensor rank \cite{teitler,pr}. 
The Fano scheme for complete intersections which are general with respect to the property of containing a fixed linear space can be used to obtain identifiability results in machine learning \cite{machine}.
Finally, the study of linear subspaces of projective toric varieties leads to a better understanding of $\A$-discriminants \cite{toric,ito}.

\subsection{Summary of Results}
In this article, we strike a balance between the general and special, and study Fano schemes for general complete intersections $X$ in a fixed projective toric variety $Y$. 
First we recall the toric situation.
Let $\A\subset \ZZ^m$ be a finite collection of lattice points and $Y=Y_\A\subset \PP^n$ be the toric variety parametrized by the monomials corresponding to elements of $\A$; here $n=\#\A-1$. 
The main result of \cite{toric} states that irreducible components $Z_{\pi,k}$ of $\bF_k(Y_\A)$ are in bijection with \emph{maximal Cayley structures} $\pi$ of length at least $k$, see \S\ref{sec:cayley}.

Consider $X\subset Y_\A$ a complete intersection cut out by Cartier divisors of classes $\alpha_1,\ldots,\alpha_r\in\Pic Y_\A$. Then $\bF_k(X)$ is a subscheme of $\bF_k(Y_\A)$, so we may study it by splitting it up into the pieces $V_{\pi,k}$ contained in each irreducible component $Z_{\pi,k}$ of $\bF_k(Y_\A)$.
We combinatorially associate an integer \[\phi=\phi(\A,\pi,\balpha,k)
\] to the data $(\A,\pi,\balpha,k)$ where: 
\begin{itemize}
\item $\A$ is a finite collection of lattice points in $\ZZ^m$, 
\item $\pi$ is a maximal Cayley structure, 
\item $\balpha =(\alpha_1,\ldots,\alpha_r)$ is a collection of Cartier divisor classes, and 
\item $k$ is a natural number which will be the dimension of the planes in our Fano scheme,
\end{itemize}
see \eqref{eqn:expected}.
The integer $\phi$ is called the \emph{expected dimension} of the Fano scheme associated to this data.
 If $V_{\pi,k}$ is non-empty, $\phi$ gives a lower bound on its dimension, and if the divisors cutting out $X$ are basepoint free and sufficiently general, $\dim V_{\pi,k}=\phi$, see Theorem \ref{thm:expected}.

Under more restrictive conditions, we are able to show that if $\phi\geq 0$, $V_{\pi,k}$ is non-empty, see Theorem \ref{thm:empty} and Corollary \ref{cor:empty}. 
Under these conditions, if $X$ is chosen to be sufficiently general, then $V_{\pi,k}$ will be smooth and of the expected dimension (Corollary \ref{cor:xsmooth}).
Along the way, we show that if $Y_\A$ is a nonsingular toric variety, the Fano scheme $\bF_k(Y_\A)$ is also nonsingular (Corollary \ref{cor:smooth}). 

In the special case that $\dim V_{\pi,k}=0$, the smoothness of $V_{\pi,k}$ allows us to use intersection theory to count the number of $k$-planes contained in $X$. Indeed, $V_{\pi,k}$ is the zero locus of a section of a vector bundle $\V$; this bundle $\V$ is a vector bundle on a relative Grassmannian, which is itself defined over a toric variety. The number of $k$-planes is the integral of the top Chern class of $\V$ (Theorem \ref{thm:chern}). This can be calculated explicitly using Schubert calculus and intersection theory on toric varieties.

\begin{ex}[See Example \ref{ex:final} for details]\label{ex:first}
We consider $Y=\PP^2\times \PP^2\subset \PP^{8}$ in its Segre embedding. The set $\A$ is just the product of two standard $2$-dimensional simplices. There are exactly two maximal Cayley structures $\pi_1$ and $\pi_2$, given by the two natural projections. 

Consider a hypersurface $X\subset Y$ cut out by a divisor of multidegree $(3,3)$ and its Fano scheme of lines $\bF_1(X)$. Then the expected dimensions for $V_{\pi_1,1}$ and $V_{\pi_2,1}$ are both zero.
For general $X$ these schemes are non-empty and smooth of the expected dimension, so $X$ contains a finite number of lines.

Each $Z_{\pi_i,1}$ is the Grassmann bundle  $\Gr(2,\E)$, where $\E=\CO_{\PP^2}(-1)^3$.
The degree of $V_{\pi_i,1}$ is the integral of the top Chern class of $\Sym^3 \cS^*$, where $\cS$ is the tautological subbundle on $\Gr(2,\E)$. A calculation with \texttt{Schubert2} \cite{schubert2}
shows that this number is $189$; it follows that the number of lines in $X$ is $378=2\cdot 189$.
The variety $X$ is a Calabi-Yau threefold; both counts of $178$ lines on $X$ are  examples of  Gromov-Witten invariants for $X$.

\end{ex}
\subsection{Organization and Acknowledgements}
We being in \S \ref{sec:prelim} by fixing notation and recalling basics about toric varieties and Cayley structures. After discussing how divisors on toric varieties restrict to linear subspaces, we discuss expected dimension.
Similar to the case of complete intersections in projective space, our approach here is based on viewing the pieces $V_{\pi,k}$ of the Fano scheme as fibers in a projection from a certain incidence scheme.

The hardest work is done in \S \ref{sec:smooth} where we prove that under certain hypotheses, our Fano schemes are non-empty and smooth. We prove this by analyzing the normal bundle of a particular linear space $L$ as we vary the equations of the complete intersection $X$. To that end, we describe the normal 
bundle of $L$ in the ambient toric variety $Y_\A$, and discuss the map taking its sections to sections of the restriction to $L$ of the normal bundle of $X$ in $Y_\A$.

Finally, in \S \ref{sec:count}. we use intersection theory to count the number of $k$-planes contained in $X$. We first describe the universal bundles on irreducible components of the Fano schemes of $Y_\A$. We then show how to realize $V_{\pi,k}$ as a section of a vector bundle, enabling us to use intersection theory to count $k$-planes.
\\
\\
\noindent The first author was partially supported by an NSERC discovery grant. The second author was supported by the Engineering and Physical Sciences Research Council under grants EP/N004922/1 and EP/S03062X/1. This project began during the Fields Institute's Thematic Program on Combinatorial Algebraic Geometry, from which both authors received partial support.

\section{Preliminaries and Expected Dimension}\label{sec:prelim}
\subsection{Toric Varieties and Cayley Structures}\label{sec:cayley}
We will always be working over an algebraically closed field $\KK$ of characteristic zero. The assumption on the characteristic is necessary since we will occasionally be applying Bertini-type results. Fix a lattice $M\cong \ZZ^m$; we denote its dual lattice $\Hom(M,\ZZ)$ by $N$. To a finite subset $\A\subset M$, we associate the projective toric variety
\[Y_\A=\Proj \KK[S_\A]\subset \PP^{\#\A-1}\]
where $S_\A$ is the semigroup generated by elements $(u,1)\in M\times \ZZ$ for $u\in \A$, and $\KK[S_\A]$ is the corresponding semigroup algebra.

Given $v\in \A$, we denote the associated homogeneous coordinate of $Y_\A$ by $x_v$. Likewise, given $u\in M$, the associated rational function on $Y_\A$ is denoted by $\chi^u$.
The variety $Y_\A$ comes equipped with an action of the torus $T=\spec \KK[M]$; the action on the projective coordinate $x_v$ has weight $v$.
We will always assume that $M$ is generated by differences of elements of $\A$, which guarantees that this action is faithful. 

We denote the inner normal fan of $\conv \A$ by $\Sigma$; this is a fan in $N_\RR=N\otimes\RR$. The abstract toric variety $Y_\Sigma$ associated to this fan as in \cite[\S3.1]{cls} is the normalization of $Y_\A$.
For more details on toric varieties see \cite{cls}.

A \emph{face} $\tau$ of $\A$ is the intersection of $\A$ with a face of $\conv \A$, and we write $\tau \prec \A$. Note that we consider $\A$ to be a face of itself. 
There is a natural closed embedding $Y_\tau\subset Y_\A$ determined by the homomorphism $\KK[S_\A]\to \KK[S_\tau]$ which for any $v\in \A$ sends
\begin{equation}\label{eqn:incl}
x_v\mapsto \begin{cases}
	x_v& v\in \tau\\
	0 &v\notin \tau
\end{cases}.
\end{equation}

We recall the main result of \cite[\S3]{toric}.
\begin{defn}[Definition 3.1 of \cite{toric}]\label{defn: Cayley structure}
A \emph{Cayley structure} of length $\ell$ on $\A$ is a surjective map $\pi:\tau\to \Delta_\ell$ preserving affine relations, where $\tau\prec \A$  and $\Delta_\ell$ is the set of standard basis vectors $e_0,\ldots,e_\ell$ in $\ZZ^{\ell+1}$.
\end{defn}
We will identify any two Cayley structures differing only by a permutation of the elements of $\Delta_\ell$.

\begin{rem}
	We emphasize that sets such as $\A$ and $\tau$, as well as $\Delta_{\ell}$ are always \emph{finite sets of lattice points} and should not be confused with their convex hulls.
\end{rem}

Any Cayley structure $\pi:\tau\to \Delta_\ell$ determines an $\ell$-dimensional linear subspace $L_\pi\subset Y_\tau\subset Y_\A$
via the surjective ring homomorphism $\KK[S_\tau]\to \KK[y_0,\ldots,y_\ell]$ sending 
\begin{equation}\label{eqn:pi}
x_v\mapsto y_{\pi(v)}.
\end{equation}
Given a $k$-dimensional linear space $L\subset \PP^n$, we denote by $[L]$ the corresponding point of $\Gr(k+1,n+1)$.

There is a natural partial order on the set of Cayley structures on $\A$ \cite[\S 3]{toric}. For Cayley structures $\pi:\tau\to\Delta_{\ell}$ and $\pi':\tau'\to\Delta_{\ell'}$, we say $\pi \geq \pi'$ if and only if $\tau'\subseteq \tau$, and there is a map $\eta:\Delta_{\ell}\to \Delta_{\ell'}$ such that $\pi'=(\eta\circ\pi)_{|\tau'}$. 
Of special importance are Cayley structures that are \emph{maximal} with respect to this partial order. These maximal Cayley structures correspond to irreducible components of toric Fano schemes.

\begin{thm}[{\cite[Theorem 3.4]{toric}}]\label{thm:toric}
	There is a bijection between irreducible components of $\bF_k(Y_\A)$ and maximal Cayley structures $\pi$ of length $\ell\geq k$. Considered with reduced structure, the irreducible component $Z_{\pi,k}$ corresponding to $\pi$ consists of the $T$-orbits of $[L]\in\Gr(k+1,n+1)$, where $L$ is any $k$-dimensional linear space contained in $L_{\pi'}$, and $\pi'$ ranges over all Cayley structures of length at least $k$ with $\pi'\leq \pi$.  
\end{thm}

\begin{figure}
	\begin{tikzpicture}
		\draw[color=lightgray,thick] (3,3)--(5,3)--(3,5)--(3,3);
		\draw[color=lightgray,thick] (0,0)--(1,0)--(0,1)--(0,0);
		\draw[color=lightgray,thick] (2,.5)--(3,.5)--(2,1.5)--(2,.5);
		\draw[color=lightgray,thick] (-.3,2)--(-.3,3)--(.7,2)--(-.3,2);
		\begin{scope}[shift={(1,0)}]\draw[red] (0,0) -- (2,.5) --(-.3,2) -- (0,0);\end{scope}
		\begin{scope}[shift={(0,0)}]\draw[green] (0,0) -- (2,.5) --(-.3,2) -- (0,0);\end{scope}
		\begin{scope}[shift={(0,1)}]\draw[blue] (0,0) -- (2,.5) --(-.3,2) -- (0,0);\end{scope}
				\foreach \position in {(0,0),(1,0),(0,1),(0,0),(2,.5),(3,.5),(2,1.5),(-.3,2),(-.3,3),(.7,2)}  \draw[fill] \position  circle[radius=.07cm];
\foreach \position in {(3,3),(5,3),(3,5),(3,4),(4,3),(4,4)}  \draw[fill,color=gray] \position  circle[radius=.07cm];
\foreach \position in {(0,0),(2,.5),(-.3,2)} \draw[lightgray,dashed] (3,3) -- \position;
\foreach \position in {(1,0),(3,.5),(.7,2)} \draw[lightgray,dashed] (5,3) -- \position;
\foreach \position in {(0,1),(2,1.5),(-.3,3)} \draw[lightgray,dashed] (3,5) -- \position;
\end{tikzpicture}

	\caption{The set $\A$ for $\Bl_{\PP^2}\PP^5$}\label{fig:A}
\end{figure}

\begin{ex}[$\Bl_{\PP^2}\PP^5$]\label{ex:1}
Consider the set $\A\subset \ZZ^5$ consisting of elements 
$(u_1,\ldots,u_5)$ satisfying $u_1,u_2,u_3\geq 1$, $u_4,u_5\geq 0$, and $\sum u_i\leq 2$. In other words, the elements of $\A$ are the columns of the following matrix:
\[
\left({\begin{array}{ccccccccccccccc}
2&0&0&1&1&0&1&0&0&1&0&0&1&0&0\\       
0&2&0&1&0&1&0&1&0&0&1&0&0&1&0\\       
0&0&2&0&1&1&0&0&1&0&0&1&0&0&1\\       
0&0&0&0&0&0&0&0&0&1&1&1&0&0&0\\
0&0&0&0&0&0&0&0&0&0&0&0&1&1&1
\end{array}}\right).
\]
The corresponding toric variety $Y_\A$ is the blowup of $\PP^5$ in a plane, with embedding in $\PP^{14}$ given by the full linear system of $2H-E$; here $H$ is the hyperplane class on $\PP^5$ and $E$ is the exceptional divisor of the blowup.

There are exactly two maximal Cayley structures on $\A$. The first is the map $\pi_1:\A\to \Delta_3$ induced by identifying the image of the map $u\mapsto (u_1+u_2+u_3,u_4,u_5)$ with a standard simplex.
The second Cayley structure is the map $\pi_2:\tau \to \Delta_2$, where $\tau$ is the facet of $\A$ with $u_1+u_2+u_3=1$, and $\pi_2$ is the projection onto the first three coordinates.
In particular, we see that $\bF_1(Y_\A)$ has two irreducible components: $Z_{\pi_1,1}$ and $Z_{\pi_2,1}$.

A projection of the set $\A$ is pictured in Figure \ref{fig:A}, with vertices and edges drawn. Each vertex of $\tau$ is shaded black. The fibers of $\pi_1$ consist of the four triangles with solid gray edges; the fibers of $\pi_2$ consist of the red, green, and blue triangles. The remaining edges are dashed lines.
\end{ex}

\begin{warning}\label{warning:reducibility}
	The irreducible components $Z_{\pi,k}$  described in Theorem \ref{thm:toric} are always considered as subvarieties of $\bF_k(Y_\A)$, that is, they are taken to be reduced. In general, however, the irreducible components of $\bF_k(Y_\A)$ might be non-reduced; we will denote them by $\widehat{Z}_{\pi,k}$. We will see later in Corollary \ref{cor:smooth} that if $Y_\A$ is smooth, $\bF_k(Y_\A)$ must also be smooth, so $\widehat{Z}_{\pi,k}=Z_{\pi,k}$.
\end{warning}

\begin{rem}
In \cite{hwang}, families of so-called minimal rational curves on a complete toric manifold are shown to be in bijection with certain primitive collections of rays of the corresponding fan. Since lines are minimal rational curves, this provides an alternate approach to understanding the lines contained in a smooth projective toric variety.
\end{rem}

\subsection{Restricting Divisors}\label{sec:restrict}
We now fix the set $\A\subset M$, along with a Cayley structure $\pi:\tau \to \Delta_\ell$. Given a $T$-invariant Cartier divisor $D$ on $Y_\A$, we are interested in understanding how $D$ (and its sections) restrict to linear subspaces of $L_\pi$. 
In order to obtain a well-defined Cartier divisor on $L_\pi$, we need that $L_\pi$ is not contained in the support\footnote{Recall that the support of $D$ consists of the union of those codimension-one subvarieties $Z\subset Y_\A$ for which the local equation of $D$ is not a unit in $\CO_{Y_\A,Z}$} of $D$ \cite[21.4]{ega}.

This is satisfied if the following assumption is met:
\begin{equation*}\label{eqn:dagger}
	Y_\tau \text{ is not contained in the support of }D.\tag{$\dagger$}
\end{equation*}
Indeed, since $D$ is $T$-invariant and by construction $L_\pi$ is not contained in the toric boundary of $Y_\tau$, \eqref{eqn:dagger} implies that $L_\pi$ is not contained in the support of $D$. Assumption \eqref{eqn:dagger} can always be achieved by replacing $D$ with another $T$-invariant divisor that is linearly equivalent to it. We remark that the divisor $D$ (nor its replacement) need not be effective.

Let $M_\tau$ be the sublattice of $M$ spanned by differences of elements of $\tau$.
Likewise, let $M_\ell$ be the sublattice of $\ZZ^{\ell}$ consisting of elements whose coordinate sum is zero. Then $\pi$ induces a surjection of lattices $\pi':M_\tau\to M_\ell$ which sends
\[(v-v')\mapsto\pi(v)-\pi(v'),\qquad v,v'\in\tau.\] 
If $\tau=\A$, then under our assumptions $M_\tau$ is just all of $M$.

Each vertex $v$ of $\A$ corresponds to a torus fixed point of $Y_\A$. 
The $T$-invariant Cartier divisor $D$ is given by a rational function of the form $\chi^{\mcd{v}}$ on a $T$-invariant neighborhood of this fixed point for some $\cd{v}\in M$.\footnote{Equivalently, the sheaf $\CO(D)$ is locally generated by $\chi^{\cd{v}}$.}
The assumption \eqref{eqn:dagger} is equivalent to requiring that for all vertices $v$ of $\tau$, $\cd{v}\in M_\tau$.
Indeed,
local equations for $D$ at $T$-invariant points are always $T$-eigenfunctions, and 
the units in the local ring of $Y_\A$ at the generic point of $Y_\tau$ which are simultaneously $T$-eigenfunctions are 
exactly elements of the form $c\cdot \chi^u$ for $c\in \KK^*$ and $u\in M_\tau$.

\begin{prop}\label{prop:restrict}
	Assume that the $T$-invariant Cartier divisor $D$ on $Y_{\mathcal{A}}$ above satisfies \eqref{eqn:dagger}.
\begin{enumerate}\item The restriction of $D$ to $L_\pi$ is given as follows: for any $e_i\in\Delta_\ell$, the local equation of $D_{|L_\pi}$ at the corresponding torus fixed point is $\chi^{-\pi'(\cd{v})}$ for any $v\in\pi^{-1}(e_i)$.
		\item Given a $T$-invariant section $\chi^u$ of $\CO(D)$ for $u \in M$, its restriction in $\CO_{L_\pi}(D)$ is
$\chi^{\pi'(u)}$ if $u \in M_\tau$, and 
 0 if $u \notin M_\tau$.
\end{enumerate}
\end{prop}
\begin{proof}
The embedding of $L_\pi$ in $Y_\A$ is given on the level of homogeneous coordinate rings by the composition 
\[
	\KK[S_\A]\to \KK[S_\tau]\to \KK[y_0,\ldots,y_\ell]
\]
where the first map is as in \eqref{eqn:incl} and the second is as in \eqref{eqn:pi}. On the other hand, local equations for $D_{|L_\pi}$ are obtained by 
pulling back local equations for $D$ via the structure map
$\iota^\#:\CO_{Y_\A}\to \iota_*(\CO_{L_\pi})$, where $\iota:L_\pi\to Y_\A$ is the inclusion. The first claim now follows by dehomogenizing the above map of homogeneous coordinate rings.
The second claim follows similarly.
\end{proof}

Recall that the divisor $D$ is basepoint free if and only if, for each vertex $w$ of $\A$, the set 
\[
	\{\cd{v}\ |\ v\in\A\ \text{ a vertex}\}\subset M
\]
is contained in $w+\pos (\A-w)$. Here $\pos$ denotes the positive hull.
In particular, the divisor $D$ can be recovered from the polytope
\[
	P_D:=\conv	\{\cd{v}\ |\ v\in\A\ \text{ a vertex}\},
\]
as each $\cd{v}$ is the unique vertex $w$ of $P_D$ 
for which $\pos (P_D-w)$ is contained in $\pos (\conv \A-v)$. We say then that $D$ is the divisor associated to $P_D$.
See \cite[\S6.1]{cls} for details.\footnote{Loc.~cit.~only covers the normal case, but it is straightforward to check that the necessary claims also apply in our setting.}

\begin{cor}\label{cor:restrict}
Assume that $D$ is basepoint free and satisfies \eqref{eqn:dagger}. Then $D_{|L_\pi}$ is the divisor associated to the polytope \[\conv \pi'(P_D\cap M_\tau).\]
In particular, the degree of $D_{|L_\pi}$ is the unique integer $\delta$ such that  $\conv \pi'(P_D\cap M_\tau)$ is lattice equivalent to $\delta\cdot \conv \Delta_\ell$.
\end{cor}

\begin{proof}
	The restriction of a basepoint free divisor is basepoint free, so the first claim follows by applying Proposition \ref{prop:restrict} at the vertices of $\tau$. The claim regarding the degree follows by noticing that the divisor on $L_\pi$ associated to $\delta\cdot \Delta_\ell$ has degree $\delta$.
\end{proof}

We say that a Cartier divisor class $\alpha\in\Pic Y_\A$ \emph{restricts surjectively} with respect to $\pi$
if the associated map $\CO(\alpha)\to \CO_{L_{\pi}}(\alpha)$
of sheaves is surjective on global sections. Analogously, we say $\alpha\in\Pic Y_\A$ restricts surjectively with respect to $\tau$ if the associated map $\CO(\alpha) \to \CO_{Y_\tau}(\alpha)$ of sheaves is surjective on global sections.
We may combinatorially check if $\alpha$ restricts surjectively using Proposition \ref{prop:restrict}, since $\alpha$ can be represented by a $T$-invariant divisor class $D$ satisfying \eqref{eqn:dagger} and $T$-invariant sections provide bases for
$H^0(Y_\A,\CO(D))$ and $H^0(L_\pi, \CO_{L_\pi}(D))$.
In particular, we have the following:
\begin{lemma}\label{lemma:surjects}
	Let $\alpha\in\Pic Y_\A$ be basepoint free. Then $\alpha$ restricts surjectively with respect to $\pi$.
\end{lemma}
\begin{proof}
	There is a face $\sigma$ of $\tau$ for which $\pi_{|\sigma}$ is bijective; this follows from e.g.~\cite[Proposition 4.3]{toric}.
	For any $e_i\in \Delta_\ell$, let $v_i\in \sigma$ be the unique element with $\pi(v_i)=e_i$. 
	Let $D$ be a $T$-invariant divisor representing $\alpha$ which satisfies \eqref{eqn:dagger}. 
	As above, for each vertex $v$ of $\A$, let $\cd{v}\in M$ be such that $\chi^{\mcd{v}}$ is a local equation for $D$. Then in particular, the lattice points $\cd{v_i}$ correspond to global sections of $\CO(D)$. Using Proposition \ref{prop:restrict} and the fact that $D$ is basepoint free, we obtain that the images of $\cd{v_i}$ under $\pi'$ are the vertices of the $\delta$th dilate of a standard simplex for some $\delta\geq 0$. 
	In particular, $D_{|L_\pi}$ corresponds to a simplex $\Delta'$ which is a lattice translate of $\delta\cdot \conv \Delta_{\ell}$.
	
	Lifting back to the polytope $P_D$ corresponding to $D$, we obtain
	that $\cd{v_i}=\delta\cdot (v_i-v_0)+\cd{v_0}$ for any $i$. The convex hull of these $\cd{v_i}$ is again a dilated simplex $P'\subset P_D$; it contains the lattice points $\cd{v_0}+\sum_j\lambda_j (v_i-v_0)$ for $\lambda_j\in\ZZ_{\geq 0}$ and $\sum \lambda_j\leq \delta$. These all correspond to global sections of $\CO(D)$. 
We again use  Proposition \ref{prop:restrict} to understand the image of $H^0(Y_\A,\CO_{Y_\A}(D))$ 
	in  $H^0(L_\pi,\CO_{L_\pi}(D))$.
The lattice points $\Delta'\cap M_\ell$ are surjected under $\pi'$ to the lattice points  $P'\cap M_\tau$ from above, that is, the $\cd{v_0}+\sum_j\lambda_j (v_i-v_0)$.
This shows that the map of global sections is surjective.
\end{proof}

\begin{ex}[$\Bl_{\PP2} \PP^5$]\label{ex:2}
	We continue the example of $Y_\A=\Bl_{\PP^2}\PP^5$ from Example \ref{ex:1}.
A divisor $D$ of class $2H-E$ is very ample, and a corresponding polytope $P_D$ is just the convex hull of $\A$ itself.
By 
Corollary \ref{cor:restrict}, we obtain that
\[
	\deg D_{|L_{\pi_1}}=1;\qquad \deg D_{|L_{\pi_2}}=1.
\]
A divisor of class $H$ is basepoint free, and a corresponding polytope $P_H$ is the convex hull of $0$ with the standard basis vectors $e_1,\ldots,e_5$. 
By 
Corollary \ref{cor:restrict}, we now obtain that
\[
	\deg H_{|L_{\pi_1}}=1;\qquad \deg H_{|L_{\pi_2}}=0.
\]
For the  computation of $\deg H_{|L_{\pi_2}}$, we note that $P_H\cap M_\tau$ consists only of $0$. By linearity, we then have that
\[
	\deg E_{|L_{\pi_1}}=1;\qquad \deg E_{|L_{\pi_2}}=-1.
\]
\end{ex}

\subsection{Expected Dimension}\label{sec:phi}
Fix $\A$ as above with a maximal Cayley structure $\pi:\tau\to \Delta_\ell$.
Let $\balpha=(\alpha_1,\ldots,\alpha_r)\in \Pic (Y_\A)^r$ be an $r$-tuple of effective non-trivial Cartier divisor classes. 
For each $i$, we define the \emph{restriction degree} of $\alpha_i$ as 
\[
	\delta_i=\deg ({\alpha_{i}}_{|L_\pi}).
\]
By choosing $T$-invariant representatives $D_1,\ldots,D_r$ of the $\alpha_i$ whose support does not contain $Y_\tau$, we may compute the $\delta_i$ combinatorially.
In the basepoint free case, we may use Corollary \ref{cor:restrict}. In general, we apply Proposition \ref{prop:restrict}. Indeed, for $v\in \tau$, let $\cd{v}$ be as in \S\ref{sec:restrict} for the divisor $D_i$. The images of these $\cd{v}$ under $\pi'$ form the vertices of a simplex $\Delta$ in $M_{\ell}$. The restriction degree $\delta_i$ is the unique integer such that $\Delta$ is lattice equivalent to $\delta_i\cdot \conv \Delta_\ell$.

\begin{rem}\label{rem:canon}
	Let $D_i$ be a torus invariant divisor of class $\alpha_i$ whose support is disjoint from $Y_\tau$, that is, satisfies \eqref{eqn:dagger}.
	For any $[L]$ in $Z_{\pi,\ell}$, we have a \emph{canonical} isomorphism 
	$\CO_{L}((D_i)_{|L})\to \CO_L(\delta_i)$. 
	For $L=L_\pi$, this is obtained by mapping $\chi^u$ as in Proposition \ref{prop:restrict} to $0$ or $\chi^{\pi'(u)}$. For $L$ in the torus orbit of $L_\pi$, this is obtained by acting on this map by the torus. 
	For $L$ outside of this orbit, we may proceed by replacing $\pi$ by an appropriate Cayley structure defined on a proper face of $\tau$.
	More precisely, let $\hat\tau$ be the minimal face of $\tau$ with $L$ contained in $Y_{\hat\tau}$, and $\hat\pi:\hat\tau\to \Delta_\ell$ the restriction of $\pi$ to $\hat\tau$. Then $L$ is in the torus orbit of $L_{\hat\pi}$, and we may proceed as above.

	In a similar fashion, one obtains a canonical isomorphism for any $L'$ in $Z_{\pi,k}$ ($k\leq \ell$)
by first considering the isomorphism above for some $L$ containing $L'$, and then restricting further.
\end{rem}

Fix $k\leq \ell$. We say that $\balpha$ satisfies \eqref{eqn:ddagger} if:
\begin{equation*}\label{eqn:ddagger}
	\begin{array}{l}
		\text{For all Cayley structures $\pi'\leq \pi$ of length at least $k$ and for all}\\\text{$i=1,\ldots,r$, the class $\alpha_i$ restricts surjectively with respect to $\pi'$.}
\end{array}
	\tag{$\dagger\dagger$}
\end{equation*}
We may check combinatorially using Proposition \ref{prop:restrict} if \eqref{eqn:ddagger} is satisfied; by Lemma \ref{lemma:surjects} it is always satisfied if all $\alpha_i$ are basepoint free.

\begin{defn}
For $X\subset Y_\A$ a complete intersection of type $\balpha$, let $V_{\pi,k}$ be the intersection of $\bF_k(X)$ with $Z_{\pi,k}$. 
Likewise, let $\widehat{V}_{\pi,k}$ be the intersection of $\bF_k(X)$ with $\widehat{Z}_{\pi,k}$ as defined in Warning~\ref{warning:reducibility}.
\end{defn}
Then $\bF_k(X)$ is the union of all $\widehat{V}_{\pi,k}$ as $\pi$ ranges over all maximal Cayley structures.
Furthermore, $V_{\pi,k}$ and $\widehat{V}_{\pi,k}$ agree when one ignores the non-reduced structure. In particular, they have the same dimension.

For $k\in \NN$, we define the \emph{expected dimension} of the configuration consisting of $\A,\pi,\balpha,k$ to be
\begin{equation}\label{eqn:expected}
\phi(\A,\pi,\balpha,k):=\dim \tau -\ell +(k+1)(\ell-k) -\sum_{i=1}^r{{k+\delta_i}\choose{k}}.
\end{equation}
Here we use the convention that ${a\choose b}=0$ if $a<b$; this occurs above if and only if $\delta_i<0$.

\begin{thm}\label{thm:expected}
Let $X\subset Y_\A$ be a complete intersection of type $\balpha$, and set $\phi=\phi(\A,\pi,\balpha,k)$.
\begin{enumerate}
\item If $V_{\pi,k}$ is non-empty, then its dimension is at least $\phi$.
\item\label{i2} If $V_{\pi,k}$ is non-empty, $X$ is sufficiently general, \eqref{eqn:ddagger} is satisfied, and  $\phi\geq 0$, then \[\dim V_{\pi,k}=\phi.\]

\item\label{i3} If $X$ is sufficiently general, \eqref{eqn:ddagger} is satisfied, and $\phi<0$, then $V_{\pi,k}$ empty.

\end{enumerate}
\end{thm}

Before the proof of this theorem, we introduce notation which will be useful later.
Let 
$\Phi=\Phi(\A,\pi,\balpha,k)$ be the incidence scheme of all tuples \[(D_1,\ldots,D_r,[L])\in |\alpha_1|\times \cdots\times |\alpha_r|\times Z_{\pi,k}\]
such that $L\subset D_1\cap \cdots \cap D_r \subseteq Y_\A$. Here, $|\alpha_i|$ denotes the linear system of all effective divisors of class $\alpha_i$.
The incidence scheme $\Phi$ comes with projections
\begin{equation}\label{eqn:projectionsFromIncidenceScheme}
\xymatrix{
\Phi \ar[d]_{\pr_1} \ar[r]_{\pr_2}&Z_{\pi,k}\\ 
|\alpha_1|\times\cdots\times|\alpha_r|   
}.\end{equation}

\begin{lemma}\label{lemma:phi}
The incidence scheme $\Phi$  has dimension at least
\[
\dim \tau -\ell +(k+1)(\ell-k) +\sum_{i=1}^r\dim |\alpha_i|-{{k+\delta_i}\choose{k}}.
\]
If \eqref{eqn:ddagger} is satisfied, it is an irreducible variety of exactly this dimension. If furthermore \eqref{eqn:ddagger} is satisfied and $Z_{\pi,k}$ is smooth, then so is $\Phi$.
\end{lemma}
\begin{proof}
We adapt the argument of \cite[Proposition 6.1]{EH} to this setting. Fix a Cayley structure $\pi'\leq \pi$ of length at least $k$.
For a fixed linear space $L$ contained in $L_{\pi'}$, consider the restriction maps 
\[
	\bigoplus_{i=1}^r H^0(Y_\A,\CO(\alpha_i))\to 
	\bigoplus_{i=1}^r H^0(Y_\A,\CO_{L_{\pi'}}(\delta_i))\to 
	\bigoplus_{i=1}^r H^0(Y_\A,\CO_{L}(\delta_i)) .
\]
The kernel $K$ of this composition has dimension at least
\[
	\sum_{i=1}^r	\dim H^0(Y_\A,\CO(\alpha_i))-\dim 
	 H^0(Y_\A,\CO_{L}(\delta_i))) 
=\sum_{i=1}^r 1+|\alpha_i|-{{k+\delta_i}\choose{k}}.
\]
Equality holds if \eqref{eqn:ddagger} is satisfied, since then first map is surjective, and the second is always surjective, so the composition is as well. 

On the other hand, the fiber $\pr_2^{-1}([L])$ may be identified with the image of the kernel $K$ in $|\alpha_1|\times\cdots\times |\alpha_r|$, so the fiber has dimension at least
\[
\sum_{i=1}^r |\alpha_i|-{{k+\delta_i}\choose{k}}
\]
with equality if \eqref{eqn:ddagger} holds.
For an arbitrary $k$-dimensional linear space $L$ with $[L]\in Z_{\pi,k}$, $L$ is obtained from a subspace of some $L_{\pi'}$ as above after acting by $T$, so the same dimension estimate for $\pr_2^{-1}([L])$ holds.

The dimension of $Z_{\pi,k}$ is 
$\dim \tau -\ell +(k+1)(\ell-k)$ \cite[Proposition 6.1]{toric}. We have seen above that every fiber of $\pr_2$ has dimension at least 
\[
\sum_{i=1}^r |\alpha_i|-{{k+\delta_i}\choose{k}}
\]
so the dimension of $\Phi$ is at least the sum of these two quantities, proving the first claim. If \eqref{eqn:ddagger} holds, then all fibers have dimension equal to the above bound; all fibers are also irreducible, since they are products of projective spaces. Since $Z_{\pi,k}$ is projective, so is $\Phi$, so the morphism $\pr_2$ is proper. The irreducibility of $\Phi$ then follows from the irreduciblity of $Z_{\pi,k}$. Furthermore, $\Phi$ is smooth if the base $Z_{\pi,k}$ is smooth.
\end{proof}
\begin{proof}[Proof of Theorem \ref{thm:expected}]
For a fixed $X\subset Y_\A$ cut out by divisors $D_1,\ldots,D_r$, $D_i\in |\alpha_i|$, the scheme $V_{\pi,k}$ is the fiber $\pr_1^{-1}((D_1,\ldots,D_r))$. Since Lemma \ref{lemma:phi} 
states that $\Phi$ has dimension at least $\phi+\dim |\alpha_1|\times\cdots\times|\alpha_r|$, a general fiber of $\pr_1$ has dimension at least $\phi$. The first claim of the theorem follows.

For the remaining two claims, assume that  $X$ is sufficiently general and \eqref{eqn:ddagger} is satisfied. By the second part of Lemma \ref{lemma:phi}, a general fiber of $\pr_1$ has dimension exactly $\phi+\dim |\alpha_1|\times\cdots\times|\alpha_r|-\dim \pr_1(\Phi)$. If $V_{\pi,k}$ is non-empty, $\pr_1$ is dominant (and thus surjective), and the dimension of $V_{\pi,k}$ is thus $\phi$. We likewise see that if $\phi<0$, $\pr_1$ is not surjective.
\end{proof}

\begin{ex}[$\Bl_{\PP^2}\PP^5$]\label{ex:3}
	We continue the example of $Y_\A=\Bl_{\PP^2}\PP^5$ from Examples \ref{ex:1} and \ref{ex:2}.
Take $\balpha=(\alpha_1)=(8H-3E)$; this class is basepoint free since $H-E$ and $H$ are both basepoint free. For the Cayley structure $\pi_1$, we have $\delta_1=5$, and obtain the expected dimension 
\[\phi(\A,\pi_1,\balpha,1)=5-3+2(3-1)-{1+5\choose1}=0.\]
Likewise, for the Cayley structure $\pi_2$,
we have $\delta_1=3$, and obtain the expected dimension
\[\phi(\A,\pi_2,\balpha,1)=4-2+2(2-1)-{1+3\choose1}=0.\]
Since $\alpha_1$ is basepoint free, we can apply Lemma~\ref{lemma:surjects}  to see that \eqref{eqn:ddagger} is satisfied. By Theorem \ref{thm:expected}, we conclude that, for a general hypersurface $X$ of class $8H-3E$ in $Y_\A=\Bl_{\PP^2}\PP^5$, there are only finitely many lines on $X$.
\end{ex}

\begin{figure}
\begin{tikzpicture}
\draw[fill,lightgray] (1,0) -- (2,0) -- (4,1) -- (3,1) -- (1,0);
				\foreach \position in {(1,0),(0,1),(2,0),(1,1),(0,2)}  \draw[fill] \position  circle[radius=.07cm];
				\foreach \position in {(3,1),(2,2),(4,1),(3,2),(2,3)}  \draw[fill] \position  circle[radius=.07cm];
				\draw (1,0) -- (2,0) -- (1,1) -- (0,2) -- (0,1) -- (1,0);
				\draw (3,1) -- (4,1) -- (3,2) -- (2,3) -- (2,2) -- (3,1);
\draw (1,0) -- (3,1);
\draw (0,1) -- (2,2);
\draw (0,2) -- (2,3);
\draw (2,0) -- (4,1);
\end{tikzpicture}

\caption{The set $\A$ for $\Bl_P\PP^2\times \PP^1$}\label{fig:ddagger}
\end{figure}

\begin{ex}[\eqref{eqn:ddagger} is necessary]
Let $Y_1=\Bl_P\PP^2$ be the blowup of $\PP^2$ in a point, with $E$ the exceptional class and $H$ the pullback of the hyperplane class. Let $Y_2=\PP^q$, $q\geq 1$, with $F$ the hyperplane class. We consider $Y_\A=Y_1\times Y_2$ embedded in projective space via the full linear system of $2H-E+F$.
Here we are abusing notation and using $H,E,F$ to also denote the pullbacks of the respective classes to $Y_1\times Y_2$.

The set $\A\subset \ZZ^2\times \ZZ^{q+1}$ can be taken to be the product of the set \[\{(1,0),(0,1),(2,0),(1,1),(0,2)\}\subset \ZZ^2\] with $\Delta_q$. A maximal Cayley structure $\pi$ of length one is given by projecting the facet $\tau=\{(1,0),(2,0)\}\times \Delta_q$ to the first $\ZZ^2$-component.
The set $\A$ is pictured in Figure \ref{fig:ddagger} in the case $q=1$, with the facet $\tau$ shaded gray.

Take $r=1$. The class $\alpha_1$ of $D=E+F$ restricts to $L_\pi$ with degree $1$, as follows from e.g.~Proposition \ref{prop:restrict}. This can also be seen easily using Theorem~\ref{thm:piface} below since both $E$ and $F$ are prime $T$-invariant divisors. However, $\alpha_1$ does not restrict surjectively with respect to $\pi$, since the image of the space of global sections of $\CO(D)$ under restriction is one-dimensional. In particular, \eqref{eqn:ddagger} is not satisfied.  We will see that item \ref{i2} of Theorem \ref{thm:expected} fails for $q=2$ and item \ref{i3} fails for $q=1$.
Indeed, we first calculate the expected dimension to be 
\[
\phi=1+q-1-2=q-2.
\]
Next, we notice that $V_{\pi,1}$ is always non-empty. Indeed, since $X$ will be the union of $E$ with the product of $\Bl_P\PP^2$ and a hyperplane in $\PP^q$, the latter will always contain a line of $Z_{\pi,1}$. Thus, although $\phi<0$ for $q=1$, $V_{\pi,k}$ is non-empty, showing that \eqref{eqn:ddagger} is necessary for item \ref{i3}.

When $q=2$, we have that $\phi=0$. However, it is easily confirmed that $V_{\pi,1}$ is always one-dimensional, showing that \eqref{eqn:ddagger} is necessary for item \ref{i2}. 
\end{ex}
\section{Non-Emptyness}\label{sec:smooth}
\subsection{Normal Bundles for $L$ in $Y_\A$}
For a fixed Cayley structure $\pi:\tau\to \Delta_\ell$, we call a face $\sigma$ of $\tau$ a $\pi$-\emph{face} if $\pi$ is injective on $\sigma$. By \cite[Proposition 4.3]{toric}, $\pi$-faces $\sigma$ of dimension $k$ are in bijection with the  torus fixed points $[L_\sigma]$ of $Z_{\pi,k}$. The linear space $L_\sigma$ is the orbit closure of $Y_\A$ corresponding to $\sigma$. With the notation from \S\ref{sec:cayley}, this is just the subvariety $Y_\sigma$ of $Y_\A$.

We will now assume that $Y_\A$ is nonsingular. This is equivalent to assuming that for each vertex $v$ of $\A$, a subset of $\A-v$ forms a basis for $M\cong \ZZ^m$ \cite[Remark 7.2]{toric}.
Consider any face $\sigma$ of $\A$. We denote the set of facets in $\A$ containing $\sigma$ by $\F_\sigma$. Since $Y_\A$ is nonsingular, the set $\F_\sigma$ consists of exactly $m-\dim \sigma$ elements. Each such facet $F$ corresponds to a ray of $\Sigma$ and hence to a torus invariant prime divisor $D_F$, see \cite[\S4.1]{cls}. Note that if $\sigma \prec \sigma'$ then it is clear that $\F_{\sigma'} \subset \F_{\sigma}$.
Given a divisor $D$ on $Y_\A$, we denote its class in $\Pic(Y_A)$ by $[D]$. 

\begin{thm}\label{thm:piface}
	Assume that $Y_\A$ is smooth, and fix a maximal Cayley structure $\pi:\tau\to \Delta_\ell$. 
	Fix any $\ell$-dimensional $\pi$-face $\sigma'$ of $\tau$, along with a $k$-dimensional face $\sigma\prec\sigma'$. 
	\begin{enumerate}
		\item\label{claim:1} $L_\sigma$ is the complete intersection in $Y_\A$ of divisors  $D_F$ for $F\in \F_{\sigma}$.
		\item \label{claim:2}For $F\in \F_{\sigma}$, the restriction of $\CO(D_F)$ to $L_\sigma$ satisfies $\CO_{L_\sigma}(D_F)=\CO_{L_\sigma}(s)$ where 
			\begin{align*}
				s<0 &\qquad\textrm{if}\ F\in\F_\tau,\\
				s=0 &\qquad\textrm{if}\ F\in\F_{\sigma'}\setminus\F_\tau, \ or\\
				s=1 &\qquad\textrm{if}\ F\in \F_\sigma\setminus \F_{\sigma'}.
			\end{align*}
	\end{enumerate}
\end{thm}

\begin{proof}
	Claim \ref{claim:1} follows from standard facts about intersections on toric varieties, see e.g.~\cite[\S12.5]{cls}.
Indeed, the elements of $F_\sigma$ correspond to the rays $\rho$ of $\Sigma$ that are in the inner normal cone $C$ of $\conv(\A)$ at $\conv(\sigma)$. By \cite[Lemma 12.5.2]{cls} and the discussion following it, the intersection of the Cartier divisors corresponding to rays in $C$ is exactly the closure of the torus orbit corresponding to the cone $C$ under the orbit-cone correspondence. But this orbit closure is exactly $L_\sigma$, and the divisors we are intersecting are exactly the $D_F$ for $F\in \F_\sigma$.

	For Claim \ref{claim:2}, 
	 we will fix some coordinates. 
The following assumptions are possible since $Y_\A$ is smooth. We assume that $\sigma$ has vertices $e_0:=0$ and $e_1,\ldots, e_k$, and $\sigma'$ has additionally the vertices $e_{k+1},\ldots,e_\ell$. The primitive generators of the edges of $\A$ starting at $0$ are $e_1,\ldots,e_\ell$ (for edges in $\sigma'$), $a_1,\ldots, a_p$ (for edges in $\tau\setminus \sigma'$), and $b_1,\ldots,b_q$ (for edges in $\A\setminus \tau$). Furthermore, we can assume that \[e_1,\ldots,e_{\ell},a_1,\ldots,a_p,b_1,\ldots,b_q\] is a basis for $M$.

We first suppose that $F\in \F_\sigma\setminus \F_{\sigma'}$. Then there is a vertex of $\sigma'$, say $e_\ell$, not contained in $F$. We consider the edge $E=\{0,e_\ell\}$ of $\A$; this corresponds to a torus invariant line $L_E\subset L_{\sigma'}$. The intersection number of $L_E$ with $D_F$ is one, since $Y_\A$ is smooth, and $F\cap E=\{0\}$. But $L_E$ is rationally equivalent to any line in $L_\sigma$, so we conclude that the degree of $\CO_{L_\sigma}(D_F)$ is also one.

For the remaining two cases, we fix the edge $E=\{0,e_1\}\subset \sigma$ and consider the intersection number $D_F.L_E$, where $L_E$ is the torus invariant line corresponding to $E$. This intersection number will again be the degree of $\CO_{L_\sigma}(D_F)$.

We have already recorded what the primitive generators of edges of $\A$ from $0$ are. We now do the same for edges from $e_1$.
Clearly, $-e_1,e_2-e_1,\ldots,e_\ell-e_1$ are primitive generators for the edges in $\sigma'$.
By considering the face of $\conv \A$ containing $0,e_1,a_i$ or $0,e_1,b_i$ we see that the remaining edges from $e_1$ have primitive generators
\begin{align*}
a_i'&=a_i+\lambda_ie_1\qquad i=1,\ldots,p\\
b_i'&=b_i+\mu_i e_1\qquad i=1,\ldots,q
\end{align*}
for integers $\lambda_i,\mu_i\geq -1$.
The inner normal vectors for the facets of $\A$ intersecting $E$ are exactly
\[
e_i^*\ (i\neq 1),\qquad a_i^*,\qquad b_i^*
\]
for the facets containing $E$, and 
\[
e_1^*,\qquad \sum \lambda_i a_i^*+\sum \mu_ib_i^*-\sum e_i^*
\]
for the two facets intersecting $E$ properly. Here, $e_i^*,a_i^*,b_i^*$ are elements of the dual basis. 

If $F\in\F_{\sigma'}\setminus\F_\tau$, the inner normal of $F$ is $v=a_i^*$ for some $i$. Likewise, if $F\in \F_{\tau}$, then its inner normal is $v=b_i^*$ for some $i$.
In both cases, to calculate $D_F.L_E$, we shift $D_F$ by the principal divisor $-\Div \chi^{v^*}$. Using  
\cite[Proposition 4.1.2]{cls} we have that  
\[
\Div \chi^{v^*}=\begin{cases}
	D_F+\lambda_i D_{F'}+R & v=a_i^*\\
	D_F+\mu_i D_{F'}+R & v=b_i^*
\end{cases}
\]
where $F'$ is a facet intersecting $E$ properly, and $R$ is a sum of prime invariant divisors corresponding to facets that do not intersect $E$. The divisor $D_F-\Div \chi^{v^*}$ intersects $L_E$ properly, and we obtain
\[
D_F.L_E=(D_F-\Div \chi^{v^*}).L_E=\begin{cases}
-\lambda_i&v=a_i^*\\
-\mu_i&v=b_i^*\\
\end{cases}.
\]
Hence, it remains to show that $\lambda_i=0$ and $\mu_i>0$ for all $i$.

For this, we use the Cayley structure $\pi$. We identify the vertices of $\Delta_\ell$ with those of $\sigma'$. Suppose that $\pi(a_i)\neq e_0$.  Then since $\pi$ is affine linear, $\conv \tau \cap [a_i^*=1]=a_i$. But then $\pi$ can be extended to a Cayley structure of length $\ell+1$ by $w\mapsto (\pi(w),a_i^*(w))$ and identifying the image with $\Delta_{\ell+1}$. This contradicts the maximality of $\pi$, so we may assume that $\pi(a_i)=e_0$ for all $i$. Therefore $\pi$ may be interpreted as the projection of $\tau$ to the subspace spanned by $e_1,\ldots,e_{\ell}$.

 Again by affine linearity of $\pi$, we then have that $\lambda_i\leq 0$. Note that every element $w$ of $\tau$ satisfies
\[
\sum \lambda_i a_i^*(w)-\sum e_i^*(w)\geq -1
\]
along with $a_i^*(w),e_i^*(w)\geq 0$. If $\lambda_i=-1$, then we could
 project $\tau$ onto $e_0,\ldots,e_\ell,a_i$, but then $\pi$ could be extended to a Cayley structure of length $\ell+1$, violating maximality.
We conclude that $\lambda_i=0$ for all $i$.

To show that $\mu_i>0$, consider $\A'=\A\cap \langle \tau,b_i\rangle$.  
Using that all $\lambda_j=0$, every element $w$ of $\A'$ must satisfy 
\[
\sum e_j^*(w)\leq  1+\mu_ib_i^*(w).
\]
If $\mu_i\leq 0$, we can project $\A'$ onto $e_0,\ldots,e_\ell$, and $\pi$ is not maximal. Hence, we conclude that $\mu_i>0$ as desired.
\end{proof}

Continuing with notation as in the above theorem, we see that the normal bundle of $L=L_\sigma$ in $Y_\A$ is given by
\begin{equation}\label{eqn:n1}
	\N_{L/Y_\A}\cong \bigoplus_{F\in \F_{\sigma}}\CO_L(D_F)\cong \CO_L(1)^{\oplus (\ell-k)}\oplus \CO_L(0)^{\oplus (\dim \tau-\ell)}\oplus \mcR,
\end{equation}
where $\mcR$ is a direct sum of $m-\dim \tau$ line bundles of negative degree.
Indeed, this follows from \cite[Proposition-Definition 6.15]{EH} and the above theorem, since the number of facets containing $\sigma$ but not $\sigma'$ is $\ell-k$, and the number of facets containing $\sigma'$ but not $\tau$ is $\dim \tau-\ell$. 

In \cite[Corollary 7.4]{toric} it is shown that if the singular locus of $Y_\A$ has dimension less than $k$, then each component $Z_{\pi,k}$ of $\bF_k(Y_\A)$, taken in its reduced structure, is nonsingular. This result does not rule out these components being non-reduced, or possibly intersecting. However, the above theorem allows us to say more if  we assume that  $Y_\A$ is nonsingular:
\begin{cor}\label{cor:smooth}
	Let $Y_\A$ be nonsingular. Then the Fano scheme $\bF_k(Y_\A)$ is also nonsingular. In particular, $Z_{\pi,k}=\widehat{Z}_{\pi,k}$ for any maximal Cayley structure $\pi:\tau\to \Delta_\ell$ and any $k\leq \ell$.
\end{cor}
\begin{proof}
We will show that $\bF_k(Y_\A)$ is nonsingular at its torus fixed points; it follows that $\bF_k(Y_\A)$ is nonsingular everywhere.
As noted above, every fixed point of $\bF_k(Y_\A)$ is of the form $[L_\sigma]$ for $\sigma$ a $k$-dimensional $\pi$-face of some Cayley structure $\pi$. The dimension of $\bF_k(Y_\A)$ at $[L_\pi]$ is at least 
\[
	\dim Z_{\pi,k}= \dim \tau-\ell+(k+1)(\ell-k)
\]
by \cite[Proposition 6.1]{toric}. 

On the other hand, the tangent space of $\bF_k(Y_\A)$ at $[L_\sigma]$ may be identified with $H^0(L_\sigma,N_{L_\sigma/Y_\A})$, see e.g.~\cite[Theorem 6.13]{EH}. By \eqref{eqn:n1}, its dimension is also
\[\dim \tau-\ell+(k+1)(\ell-k).\] Thus, $\bF_k(Y_\A)$ is smooth at $[L_\sigma]$.
\end{proof}
\begin{rem}
	For $k\geq 2$, a slightly more straightforward proof of the non-singularity of $\bF_k(Y_\A)$ is possible. As noted above, the normal bundle 
	$\N_{L/Y_\A}$
	for $L$ at any torus fixed point is a direct sum of line bundles. But then $H^1(L,\N_{L/Y_\A})=0$, since $L\cong \PP^k$, $k\geq 2$. It follows that $\bF_k(Y_\A)$ is nonsingular at each torus fixed point, hence is nonsingular.
\end{rem}

\subsection{Normal Bundles and Cox Coordinates}
We continue under the assumption that $Y_\A$ is nonsingular.
Let  $X\subset Y_\A$ be a complete intersection of type $\balpha=(\alpha_1,\ldots,\alpha_r)$ as in \S\ref{sec:phi} and $L\subset X$ be a $k$-plane.
Then there is an exact sequence 
\begin{equation}\label{eqn:normal}
0\to \N_{L/X} \to \N_{L/Y_\A} \to \N_{X/Y_\A|L}
\end{equation}
of (restrictions of) normal bundles, see e.g.~\cite[Proposition-Definition 6.15]{EH}.
The space of global sections of $\N_{L/X}$ may be identified with the tangent space of the point $[L]$ of $\bF_k(X)$ \cite[Theorem 6.13]{EH}. We wish to show that, in favorable situations, the dimension of this tangent space agrees with the expected dimension. To that end, we need a better understanding of the normal bundles appearing in this exact sequence.

In the situation that $L=L_\sigma$ for some $\pi$-face $\sigma\prec \A$, Theorem \ref{thm:piface} gives us good control over $N_{L/Y_\A}$. On the other hand since $X$ is a complete intersection in $Y_\A$, we have
\begin{equation}\label{eqn:n2}
	\N_{X/Y_A|L}\cong \bigoplus_{i=1}^r\CO_L(\alpha_i)\cong \bigoplus_{i=1}^r\CO_L(\delta_i).
\end{equation}

The \emph{Cox ring} of $Y_\A$ is the polynomial ring \[R(Y_\A)=\KK[y_F\ |\ F\ \textrm{a facet of}\ \A],\]
see \cite[Chapter 5]{cls} for details. This ring comes with a grading by the Picard group of $Y_\A$.
For any torus invariant divisor $D$, the degree $[D]$ graded piece of $R(Y_\A)$ may be canonically identified with the space of global sections of $\CO_{Y_\A}(D)$. 
The ring of Laurent polynomials of degree zero (under the grading by $\Pic(Y_\A)$) is canonically identified with $\KK[M]$.
Furthermore, any subscheme of $Y_\A$ may be described via a homogeneous ideal of $Y_\A$. In particular, for a $\pi$-face $\sigma$, $L_\sigma$ is described by the ideal $I_\sigma$ generated by the $y_F$ for $F\in \F_\sigma$. 
Likewise, each divisor $D_i$ of class $\alpha_i$ cutting out $X$ 
corresponds to an element $g_i\in \KK[y_F]$ of degree $\alpha_i$.  
The condition that $L_\sigma$ is contained in each $D_i$ implies that $g_i$ is contained in the saturation of $I_\sigma$ by a certain monomial ideal (the so-called \emph{irrelevant ideal}), see e.g.~\cite[Proposition 6.A.7]{cls}. But $I_\sigma$ is already saturated with respect to the maximal ideal generated by all the variables $y_F$, so it is in particular saturated with respect to the irrelevant ideal. We conclude that in fact $g_i$ is contained in $I_\sigma$. Hence, we may write
\begin{equation}\label{eqn:g_i}
	g_i=\sum_{F\in \F_\sigma} y_F\cdot g_{i}^F
\end{equation}
for some $g_{i}^F\in R(Y_\A)$.

The variety $Y_\sigma\cong \PP^k$ is also toric, hence it also has a Cox ring \[
	R(Y_\sigma)=\KK[y_F\ |\ F\ \textrm{a facet of}\ \sigma
		],
\]
which may be identified with the homogeneous coordinate ring of $\PP^k$.
\begin{lemma}
The natural map $\rho:R(Y_\A)\to R(Y_\sigma)$ induced by restricting sections of line bundles from $Y_\A$ to $Y_\sigma$ is defined by 
\[
	y_F\mapsto \rho(y_F)=\begin{cases}
		0& \sigma \subseteq F\\
		1 & \sigma\cap F=\emptyset\\
		y_{F\cap \sigma} & \textrm{else}
	\end{cases}.
\]
\end{lemma}
\begin{proof}
	The monomial $y_F$ corresponds canonically to the global section $1$ of the bundle $\CO(D_F)$. As long as $\sigma$ does not intersect $F$, $D_F$ restricts to the trivial divisor on $L_\sigma$, so we obtain the global section $1$ of the bundle $\CO_{L_\sigma}$. This corresponds to the element $1$ of $R(Y_\sigma)$.
	If instead $\sigma$ intersects $F$ but is not contained in it, $L_\sigma$ is not contained in the support of $D_F$, and we may consider $(D_F)_{|L_\sigma}$. It follows from Proposition \ref{prop:restrict} that this restriction is $D_{F\cap \sigma}$, and so in this case $y_F$ maps to $y_{F\cap \sigma}$. 
	
	Finally, if $\sigma$ is contained in $F$, we instead identify $y_F$ with the section $\chi^u$ of the bundle $\CO(D_F+\Div\chi^{-u})$ for an appropriate choice of $u$ such that $D_F+\Div\chi^{-u}$ doesn't contain $L_\sigma$ in its support.
	Taking $\nu\in N$ to be the primitive generator of the inward normal of $F$, the requirement is exactly that $\nu(u)=1$. But since $\sigma\subset F$, for any $w\in M_\sigma$ we have $\nu(w)=0$. In particular, $u\notin M_\sigma$, so by Proposition \ref{prop:restrict}, $\chi^u$ restricts to $0$. It follows that $y_F$ does as well. 
\end{proof}

\begin{lemma}\label{lemma:surj2}
	Written using Cox coordinates, the map
\[
	\bigoplus_{F\in \F_\sigma} \CO_L([D_F]) \to \bigoplus_{i=1}^r\CO_L(\delta_i)
\]
induced by \eqref{eqn:normal} and the above identifications of normal bundles sends 
a  section $f$ of $\CO_L([D_F])$ to the section $(f\cdot \rho(g_i^F))_{i=1}^r$ of $\bigoplus_{i=1}^r \CO_L(\delta_i)$.
\end{lemma}
\begin{proof}
	We work locally. Fix any open affine torus invariant $U\subset Y_\A$ containing a torus fixed point. 
	For any class $\beta \in \Pic Y_\A$, there is a unique 
torus invariant divisor $D=\sum a_F D_F$ of class $\beta$ which is trivial on $U$. We write
$q_\beta=\prod_F y_F^{a_F}$.
	
For each $F\in \F_\sigma$, $D_F$ restricted to $U$ is the principal divisor of $z_F:=y_F/q_{[D_F]}$.
Likewise, each $D_i$ restricted to $U$ is the principal divisor  
of $h_i:=g_i/q_{\alpha_i}$. Setting $h_i^F=g_i^F/q_{\alpha_i-[D_f]}$, we then have
\[
	h_i=\sum_{F\in \F_\sigma} z_F\cdot h_i^F,
\]
and that the $z_F,h_i^F$ are regular functions on $U$. 

	Denote by $S$ the coordinate ring of $U$. Let $I$ be the ideal of $S$ generated by the $z_F$. Likewise, let $J$ be the ideal of $S$ generated by the $h_i$. Then $I$ is the ideal of $L$ in $U$, $J$ is the ideal of $X$ in $U$, and $J\subseteq I$. 
	The map $\N_{L/Y_\A}\to \N_{X/Y|L}$ is locally given by the natural module homomorphism
	$\Hom(I,S/I)\to \Hom(J,S/I)$
obtained by restricting from $I$ to $J$.
Viewed as an $S/I$-module, $\Hom(I,S/I)$ is free with generators $\psi_F$ for $F\in \F_\sigma$ defined via
\[
	\psi_F (z_{F'})=\begin{cases} 1 & F=F', \\ 0 & \textrm{else}.\end{cases}
\]
Likewise, $\Hom(J,S/I)$ is free with generators $\psi_i$ defined by
\[
	\psi_i (h_j)=\begin{cases} 1 & i=j, \\ 0 & \textrm{else}.\end{cases}
\]
By $S/I$-linearity, we  see that $\psi_F$ maps to $\sum_i \overline h_i^F\psi_i$, where $\overline h_i^F$ the residue class of $h_i^F$ in $S/I$.

Since
\[\rho(g_i^F)/q_{(\alpha_i-[D_F])_{|L}}
=\overline {g_i^F/ q_{\alpha_i-[D_F]}}
=\overline h_i^F,
\]
we see that after globalizing and passing back to Cox coordinates, the map of \eqref{eqn:normal} agrees with the description in the lemma.
\end{proof}

\begin{prop}\label{prop:surj}
Assume that $Y_\A$ is nonsingular, and let $L=L_\sigma$, $X$, and $g_i^F$ be as above in Equation~\eqref{eqn:g_i}.
Then the induced map
\begin{equation}\label{eqn:gn}
H^0(L,\N_{L/Y_\A}) \to H^0(L,\N_{X/Y_\A|L})
\end{equation}
is surjective if and only if the $R(L)$-submodule of $\bigoplus_{j\geq 0}H^0(L,\CO_L(j))^r$ generated by $(\rho(g_i^F))_{i=1}^r$ for $F\in \F_\sigma\setminus \F_\tau$ contains 
$ H^0(L,\CO_L(\delta_1))\oplus\cdots\oplus H^0(L,\CO_L(\delta_r)) $.

Furthermore, in this case, $\dim H^0(L,\N_{L/X})=\phi$, the expected dimension.
\end{prop}

\begin{proof}
We've identified the codomain of the map \eqref{eqn:gn}  with $ H^0(L,\CO_L(\delta_1))\oplus\cdots\oplus H^0(L,\CO_L(\delta_r)) $.
The summands of 
$\N_{L/Y_\A}\cong \bigoplus_{F\in \F_\sigma} \CO_L([D_F])$ with global sections are exactly those with $F\notin \F_\tau$ by Theorem \ref{thm:piface}. The first claim now follows from the description of the map \eqref{eqn:gn} from Lemma \ref{lemma:surj2}.

For the second claim, we use the exactness sequence of sheaves in~\eqref{eqn:normal} to obtain the exact sequence of cohomology groups 
\begin{equation*}
0\to H^0(L,\N_{L/X}) \to H^0(L,\N_{L/Y_\A}) \to H^0(L,\N_{X/Y_\A|L})\to 0
\end{equation*}
along with \eqref{eqn:n1} and \eqref{eqn:n2}
to count
\begin{align*}
\dim  H^0(L,\N_{L/X})&=\dim  H^0(L,\N_{L/Y_\A})- \dim H^0(L,\N_{X/Y_\A|L})\\&=\left(\dim \tau-\ell + (k+1)(\ell-k)\right)-\left(\sum_{i=1}^r {k+\delta_i\choose k} \right)\\&=\phi(\A,\pi,\balpha,k).
\end{align*}\end{proof}

\subsection{Criterion for Surjectivity}
When can we find $X\subset Y_\A$ such that the map of Proposition \ref{prop:surj} is surjective?
To find such $X$, we will choose $g_i^F\in R(Y_\A)_{[D_i-D_F]}$ to obtain $g_i\in R(Y_\A)$ as in \eqref{eqn:g_i}, and hence divisors $D_i$.
Note that, instead of choosing the $g_i^F$, we can focus on choosing their restrictions $\rho(g_i^F)$.
To do this, we need the following algebraic fact:
\begin{lemma}\label{lemma:surj}
	Let $a,b,\delta_i\in \ZZ_{\geq 0}$, $i=1,\ldots r$.
Assume that 
\begin{align}
	a+b-k-r\geq 0;\label{eqn:1}\\
	a(k+1)+b-\sum_{i=1}^r {\delta_i+k\choose k}\geq 0\label{eqn:2}
\end{align}	
and furthermore that one of the following three conditions hold:
\begin{enumerate}
	\item $\delta_i\geq 3$ for some $i$;
	\item $\delta_i\geq 2$ for at least two indices $i$; or
	\item  $a-k-\#\{i\ |\ \delta_i=1\}\geq 0$.
\end{enumerate}
Then the map
	\begin{align*}
	\lambda:H^0(\PP^k,\CO_{\PP^k}(1))^{\oplus a}\oplus H^0(\PP^k,\CO_{\PP^k})^{\oplus b}&\to \bigoplus_{i=1}^r H^0(\PP^k,\CO_{\PP^k}(\delta_i))\\
	((s_j)_{j=1}^a,(s'_j)_{j=1}^b)&\mapsto \left(\sum_{j=1}^a s_j\cdot h_{ij}+\sum_{j=1}^b s_j'\cdot h_{ij}'\right)_{i=1}^r\
\end{align*}
is surjective for general choice of 
$h_{ij}\in H^0(\PP^k,\CO_{\PP^k}(\delta_i-1))$ and 
$h_{ij}'\in H^0(\PP^k,\CO_{\PP^k}(\delta_i))$. 
\end{lemma}
The special case of this lemma when $b=0$ is used classically for showing the non-emptyness of Fano schemes of complete intersections in projective space, see \cite[Proof of Theorem 6.28]{EH} and \cite{hochster} for the case of hypersurfaces ($r=1$) and \cite[\S2]{debarre:98a} for complete intersections.
We adapt the argument of \cite[\S2]{debarre:98a} to include the case $b>0$:
\begin{proof}[Proof of Lemma \ref{lemma:surj}]
For ease of notation, set 
	$V=\bigoplus_{i=1}^r H^0(\PP^k,\CO_{\PP^k}(\delta_i-1))$ and
	$W=\bigoplus_{i=1}^r H^0(\PP^k,\CO_{\PP^k}(\delta_i))$.
Likewise, set
	\[
\mcH=\bigoplus_{i=1}^r H^0(\PP^k,\CO_{\PP^k}(\delta_i-1))^a\times 
	\bigoplus_{i=1}^r H^0(\PP^k,\CO_{\PP^k}(\delta_i))^b.
\]
We are trying to show that there exist $((h_{ij}),(h'_{ij}))\in\mcH$ for which $\lambda$ is surjective.
	
	Let 
\[
	\mu:H^0(\PP^k,\CO_{\PP^k}(1))\times V\to W
\]
be the natural multiplication map. For any hyperplane $H\subset W$, denote by $\mu^{-1}(H)$
the set
\[
	\mu^{-1}(H)=\{v\in V\ |\ v\cdot H^0(\PP^k,\CO_{\PP^k}(1))\subset H\}.
\]
From its definition, it is apparent that the codimension of $\mu^{-1}(H)$ is at most $k+1=\dim H^0(\PP^k,\CO_{\PP^k}(1))$.
For $t=1,\ldots,k+1$, we define the set
\begin{align*}
	\mcL_t=\{\psi \in W^*\ | \codim(\mu^{-1}(\ker \psi),V)=t\}
\end{align*}
and the subset $\mcZ_t$ of
	$\mcH\times	
	\PP(\mcL_t)
$
by
\begin{align*}
	\mcZ_t=\left\{((h_{ij}),(h_{ij}'),\overline \psi) \in \mcH\times\PP(\mcL_t)
	\ \Bigg| \begin{array}{l}
	\forall j=1,\ldots,a,\ (h_{ij})\in \mu^{-1}(\ker \psi)\\
	\forall j=1,\ldots,b,\ (h'_{ij})\in \ker \psi\\
\end{array}
\right\}.
\end{align*}

The subset $Z$ 
of $\mcH$ 
for which the map $\lambda$ is not surjective is the union of the projections of $\mcZ_t$ to $\mcH$.
We thus obtain that
\[
	\codim(Z,\mcH) \geq \min_{t=1,\ldots,k+1} \codim (\mcZ_t,\mcH\times \PP(\mcL_t))-\dim \PP(\mcL_t).
\]

On the one hand, the conditions in the definition of $\mcZ_t$ are clearly independent, with the condition $(h_{ij}\in\mu^{-1}(\ker\psi))$ having codimension $t$, and $(h_{ij}')\in\ker \psi$ having codimension $1$. Hence,
\[
	\codim (\mcZ_t,\mcH\times \PP(\mcL_t))=t\cdot a+b.
\]
On the other hand, \cite[Lemma 2.8]{debarre:98a} states that for $1\leq t \leq k+1$,
\[
	\dim \PP(\mcL_t)\leq t(k-t+1)+\sum_{i=1}^r {\delta_i+t-1\choose t-1}-1.
\]
Combining, we obtain that 
\[
	\codim(Z,\mcH) \geq \min_{t=1,\ldots,k+1} 
\gamma(t)
\]
where
\[\gamma(t):=
b+1
+	t(a+t-k-1)-\sum_{i=1}^r {\delta_i+t-1\choose t-1}.
\]

The first forward difference of $\gamma$ is 
\[\gamma(t+1)-\gamma(t)=2t+a-k -\sum_{i: \delta_i\geq 1}{\delta_i+t-1\choose t}.\]
Likewise, the second order forward difference is
\[\gamma(t+2)-2\gamma(t+1)+\gamma(t)=2-\sum_{i: \delta_i\geq 2} {\delta_i+t-1\choose t+1}.\]

If at least two $\delta_i$ are greater than or equal to two, or one $\delta_i\geq 3$, then the second forward difference is always non-positive for $t=1,\ldots,k+1$, so $\gamma$ is concave. If neither of these conditions is met, but $a-k-\#\{i\ |\ \delta_i=1\}\geq 0$, then the first forward difference is always non-negative, so $\gamma$ is non-decreasing. In either case, we can compute $\min_{t=1,\ldots,k+1} 
\gamma(t)$ by only evaluating at $t=1$ and $t=k+1$ and taking the smaller value. 

Using the inequalities \eqref{eqn:1} and \eqref{eqn:2}, we then obtain
that 
\[
\codim(Z,\mcH)\geq 1
\]
implying that for general choice of $(h_{ij})$ and $(h_{ij}')$, $\lambda$ is surjective.
\end{proof}

\subsection{The Non-Emptyness Theorem}
We continue with notation established previously. Our main result of this section is the following:
\begin{thm}\label{thm:empty}
		Let $X\subset Y_\A$ be a complete intersection of type $\balpha$, and fix a maximal Cayley structure $\pi:\tau\to \Delta_\ell$.
	Assume the following:
	\begin{enumerate}
		\item $Y_\A$ is nonsingular;\label{as:0}
		\item \eqref{eqn:ddagger} is satisfied;\label{as:1}
		\item There is an $\ell$-dimensional $\pi$-face $\sigma'$ of $\tau$ with $k$-dimensional face $\sigma\prec \sigma'$ such that for every $i=1,\ldots,r$ and every $F\in\F_\sigma\setminus \F_{\sigma'}$, $\alpha_i-[D_F]$ restricts surjectively with respect to $\sigma$;\label{as:3}
		\item $\phi(\A,\pi,\balpha,k)\geq 0$;\label{as:4}
		\item $\delta_i\geq 0$ for $i=1,\ldots,r$;\label{as:2}
		\item $\dim \tau-2k-r\geq 0$;\label{as:5}
		\item Either $\delta_i\geq 3$ for some $i$, $\delta_i\geq 2$ for at least two indices $i$, or $\ell-2k-\#\{i\ |\ \delta_i=1\}\geq 0$.\label{as:6}
	\end{enumerate}
	Then the variety $V_{\pi,k}$ is non-empty.
	Furthermore, it is scheme-theoretically a union of irreducible components of $\bF_k(X)$.
\end{thm}

\begin{proof}
	Taking 
	$a=\ell-k$, $b=\dim \tau-\ell$, assumptions \ref{as:4}, \ref{as:2}, \ref{as:5}, and \ref{as:6} allow us to apply Lemma \ref{lemma:surj} to produce 
$h_{ij}\in H^0(\PP^k,\CO_{\PP^k}(\delta_i-1))$ and 
$h_{ij}'\in H^0(\PP^k,\CO_{\PP^k}(\delta_i))$ such that the map $\lambda$ of Lemma \ref{lemma:surj} is surjective.

Let $\sigma$ and $\sigma'$ be $\pi$-faces as in assumption \ref{as:3}. 
By assumption \ref{as:0}, Theorem \ref{thm:piface} and the following discussion imply that there are $a=\ell-k$ faces $F\in \F_\sigma$ with $\CO(D_F)$ restricting to $\CO_{L_\sigma}(1)$, and $b=\dim \tau-\ell$ faces $F\in \F_\sigma$ with $\CO(D_F)$ restricting to $\CO_{L_\sigma}(1)$. 

By assumption \ref{as:3}, the map $R(Y_\A)_{\alpha_i-[D_F]}\to H^0(L_\sigma,\CO_{L_\sigma}(1))$ for $F\in \F_{\sigma}\setminus\F_{\sigma'}$ is surjective, so after ordering these faces we may lift the $(h_{ij})$ to $g_i^F\in R(Y_\A)_{\alpha_i-[D_F]}$. Likewise, by assumption \ref{as:1}, the map 
$R(Y_\A)_{\alpha_i-[D_F]}\to H^0(L_\sigma,\CO_{L_\sigma}(0))$ for $F\in \F_{\sigma'}\setminus\F_\tau$ is surjective, so after ordering these faces we may lift the $(h_{ij}')$ to $g_i^F\in R(Y_\A)_{\alpha_i-[D_F]}$. 
We now set 
\[
	g_i=\sum_{F\in \F_\sigma\setminus \F_\tau} y_F\cdot g_i^F
\]
and consider the corresponding divisors $D_i$ and complete intersection $X$.

Now, $\rho(g_i^F)$ is just $h_{ij}$ or $h_{ij}'$ for appropriate choice of index $j$, and our construction of the $h_{ij},h_{ij}'$ guarantees by Proposition \ref{prop:surj} that $\dim H^0(L,\N_{L/X})=\phi$ for $L=L_\sigma$.

The non-emptyness of $V_{\pi,k}$ for any complete intersection of type $\balpha$ now follows from Lemma \ref{lemma:fiber} below. 
For the second claim, $\widehat{V}_{\pi,k}$ is scheme-theoretically a union of irreducible components of $\bF_k(X)$, since $\widehat{Z}_{\pi,k}$ is an irreducible component of $\bF_k(Y_\A)$, and $\bF_k(X)\subseteq \bF_k(Y_\A)$. However, Corollary \ref{cor:smooth} implies that $\widehat{Z}_{\pi,k}=Z_{\pi,k}$, so
$\widehat{V}_{\pi,k}=V_{\pi,k}$.
\end{proof}

Recall from \S\ref{sec:phi} the incidence scheme $\Phi$ of all tuples $(D_1,\ldots,D_r,[L])$ in $|\alpha_1|\times \cdots\times |\alpha_r|\times Z_{\pi,k}$ such that $L\subset D_1\cap \cdots \cap D_r \subseteq Y_\A$.
The scheme $\Phi$ has projections to each component of the product $(|\alpha_1|\times\cdots\times|\alpha_r|) \times Z_{\pi,k}$, as seen in \eqref{eqn:projectionsFromIncidenceScheme}. We now provide sufficient criteria for the projection onto $|\alpha_1|\times\cdots\times|\alpha_r|$ to be surjective.

\begin{lemma}\label{lemma:fiber}
	Suppose that $Y_\A$ is nonsingular, \eqref{eqn:ddagger} holds, and that  there is some $X\subset Y_\A$ a complete intersection of type $\balpha$ and a $k$-plane $L\subset X$, $[L]\in Z_{\pi,k}$ such that 
\[
	\dim H^0(L,\N_{L/X})\leq \phi=\dim \Phi-\sum_{i=1}^r \dim |\alpha_i|.
\]
Then the map $\pr_1$ is surjective, that is, \emph{every} complete intersection of type $\balpha$ contains a $k$-plane from $Z_{\pi,k}$.
\end{lemma}
\begin{proof}
Let $B$ be the image of $\pr_1$; it is a closed subscheme of $|\alpha_1|\times\cdots\times|\alpha_r|$. Suppose that $\pr_1$ is not surjective, that is, 
$\dim B<\sum_{i=1}^r \dim |\alpha_i|$. We consider the relative cotangent sheaf $\Omega_{\Phi/B}$; let $U\subset \Phi$ consist of those points $x$ for which $\Omega_{\Phi/B,x}$ is generated by fewer than $q=\dim \Phi-\dim B$ elements.
By e.g.~\cite[Exercise II.5.8(a)]{hartshorne}, the set $U$ is an open subset of $\Phi$.

For $b\in Y$, set $\Phi_b=\pr_1^{-1}(b)$.
We claim that the closed points of $U$ may also be described as the set those $x\in \Phi$ such that $\dim T_{\Phi_b,x}<q$, where $b=\pr_1(x)$ and $T_{\Phi_b,x}$ is the tangent space of the scheme $\Phi_b$ at the point $x$.
Indeed, locally
\[
	T_{\Phi_b,x}^*\cong \Omega_{\Phi_b,x}\otimes \CO_{\Phi_b,x}/\mfm_{\Phi_b,x}\cong \Omega_{\Phi/B,x}\otimes \CO_{\Phi,x}/\mfm_{\Phi,x},
\]
where $\mfm_{\Phi_b,x}$ and $\mfm_{\Phi,x}$ denote the maximal ideal of $\Phi_b$ and $\Phi$ at $x$.
The first isomorphism follows from e.g. \cite[Proposition II.8.7]{hartshorne}, and the second from \cite[Proposition II.8.2a]{hartshorne}.
The claim regarding $U$ now follows from Nakayama's lemma.

By the hypothesis of the lemma and the interpretation of $H^0(L,\N_{L/X})$ as the tangent space of the Fano scheme at a point, the set $U$ is thus non-empty.
But since $U$ is a non-empty open set of $\Phi$, it dominates $B$, and it follows from our above description of $U$ that a general fiber $\Phi_b$ of $\pr_1$ contains a point $x$ such that the dimension of $\Phi_b$ at $x$ is 
less than $\dim \Phi-\dim B$.
Furthermore, since $Y_\A$ is nonsingular, so is $\bF_k(Y_\A)$ (Corollary \ref{cor:smooth}), and hence also $\Phi$ (Lemma \ref{lemma:phi}).
By generic smoothness (\cite[Corollary  III.10.7]{hartshorne}) the fiber $\Phi_b$ is smooth (and equidimensional), so its dimension is less than $\dim \Phi-\dim B$. But this is impossible, since it would imply $\dim \Phi <(\dim \Phi -\dim B)+\dim B=\dim \Phi$.

We conclude that $\pr_1$ must have been surjective.
\end{proof}

We can simplify the hypotheses of Theorem \ref{thm:empty} by making slightly stronger assumptions.
\begin{cor}\label{cor:empty}
		Let $X\subset Y_\A$ be a complete intersection of type $\balpha$, and fix a maximal Cayley structure $\pi:\tau\to \Delta_\ell$.
		Assume that $Y_\A$ is nonsingular, and for each torus invariant prime divisor $P$ and each $i=1,\ldots,r$, $\alpha_i$ and $\alpha_i-[P]$ are basepoint free.
		Assume further that $\phi(\A,\pi,\balpha,k)\geq 0$ and 
		$\dim \tau-2k-r\geq 0$. Finally, assume that either $\delta_i\geq 3$ for some $i$, $\delta_i\geq 2$ for at least two indices $i$, or $\ell-2k-\#\{i\ |\ \delta_i=1\}\geq 0$.
	Then the variety $V_{\pi,k}$ is non-empty.
	Furthermore, it is scheme-theoretically a union of irreducible components of $\bF_k(X)$.
	\end{cor}
	\begin{proof}
		We must show that the hypotheses for Theorem \ref{thm:empty}
		are satisfied. The basepoint freeness of the $\alpha_i$ implies assumptions \ref{as:1} and \ref{as:2} by Lemma \ref{lemma:surjects}. Similarly, the basepoint freeness of $\alpha_i-P$ implies assumption \ref{as:3}. The remaining hypotheses for the theorem are the same as in the corollary.
	\end{proof}

	\begin{cor}\label{cor:xsmooth}
		Assume that the hypotheses of Theorem \ref{thm:empty} or Corollary \ref{cor:empty} are satisfied. If $X\subset Y_\A$ is sufficiently general, then $V_{\pi,k}$ is smooth of dimension $\phi$.
		\end{cor}
		\begin{proof}
Assuming the hypotheses of Theorem \ref{thm:empty} or Corollary \ref{cor:empty}, we have in particular by Corollary \ref{cor:smooth} that $Z_{\pi,k}$, hence $\Phi$ is smooth by Lemma \ref{lemma:phi}. By Theorem \ref{thm:empty} combined with Theorem \ref{thm:expected}, a general fiber $V_{\pi,k}$ of $\pr_1$ has dimension $\phi$; by generic smoothness, it will also be smooth.
		\end{proof}

		\begin{ex}[$\Bl_{\PP^2}\PP^5$]\label{ex:4}
			We continue the example of $\Bl_{\PP^2}\PP^5$ from Examples \ref{ex:1}, \ref{ex:2}, and \ref{ex:3}. The torus invariant prime divisors have classes $H$, $E$, and $H-E$. As mentioned before, since $H-E$ and $H$ are both basepoint free, the divisor $\alpha_1=8H-3E$ is basepoint free, and so are the divisors $7H-3E$, $8H-4E$, and $7H-4E$. 
			For both $\pi_1$ and $\pi_2$ and $k=1$, the expected dimension is zero, and in both cases, $\dim \tau-2k-r\geq 0$. Since $\delta_1$ for the two Cayley structures $\pi_i$ is equal to $5$ and $3$, respectively, Corollaries \ref{cor:empty} and \ref{cor:xsmooth} apply in both cases. We conclude that for sufficiently general $X\subset \Bl_{\PP^2}\PP^3$ of class $8H-3E$, $\bF_1(X)$ is the disjoint union of $V_{\pi_1,1}$ and $V_{\pi_2,1}$, both of which are non-empty, zero-dimensional, and smooth.
				\end{ex}

	\begin{rem}
		The now classical results on Fano schemes for sufficiently general hypersurfaces and complete intersections $X\subset \PP^n$ typically include the statement that if the expected dimension is at least $1$, and $X$ is sufficiently general, then $\bF_k(X)$ is connected, see e.g.~\cite[Theorem 2.1.c]{debarre:98a}. The methods used to show this require understanding the locus of the incidence variety $\Phi$ where the first projection $\pr_1$ is smooth. In essence this is done by considering the sequence \eqref{eqn:normal} and subsequent analysis for \emph{all} possible linear subspaces $L\subset \PP^n$. This is possible since every linear subspace of $\PP^n$ is a complete intersection. By contrast, we are only able to carry out our analysis for the special $L=L_\sigma$, because general $L\subset Y_\A$ might not be a complete intersection.
	\end{rem}

\section{Counting Linear Subspaces}\label{sec:count}
In the classical setting of a degree $d$ hypersurface $X$ in $\PP^n$, $X$ determines a section of the bundle 
$\Sym^d \cS^*$, where $\cS$ is the tautological subbundle on the Grassmannian $\Gr(k+1,n+1)$. 
The Fano scheme $\bF_k(X)$ is the zero locus in $\Gr(k+1,n+1)$ of this section. In particular, if $\bF_k(X)$ has the expected dimension, then its class in the Chow ring of the Grassmannian is just the top Chern class of $\Sym^d \cS^*$. This allows one to count for example the $27$ lines on a cubic surface and $2875$ lines on a quintic threefold using intersection theory.
See e.g.~\cite[Proposition 6.4]{EH} for details.

Here, we adapt this approach to our setting.
As before, $X$ will be a complete intersection of type $\balpha$ in the toric variety $Y_\A$. We fix a maximal Cayley structure $\pi:\tau\to \Delta_\ell$.
We will be interested in $V_{\pi,k}\subseteq \bF_k(X)$.

	Set 
	\[\A_\pi=\{u_0+\ldots+u_\ell \in M\ |\ u_i\in\pi^{-1}(e_i)\cap \tau\}.\]
The  variety $Z_\pi=Z_{\pi,\ell}$, considered in its reduced structure, is the projective toric variety $Y_{\A_\pi}$ \cite[Theorem 6.2]{toric}. 
The natural torus acting on $Z_\pi$ is the torus $T_\pi$ whose character lattice $M_\pi$ is generated by differences of elements of $\A_\pi$.

In what follows, we will encode a number of globally generated line bundles $\mcL_i$ on $Z_\pi$ by giving a subset  $\G_i\subset M_\pi$ such that as $w$ ranges over the elements of $\G_i$, $\chi^w$ are the local generators of $\mcL_i$. Since the bundles $\mcL_i$ are globally generated, they are already determined by the set $\G_i$: for any open $U\subset Z_\pi$, \begin{equation}\label{eqn:sheaf}\mcL_i(U)=\CO_{Z_\pi}(U)\cdot\{\chi^w\ |\ w\in \G_i\}\subset \KK(Z_{\pi}).\end{equation}
\begin{rem}
	Given an \emph{arbitrary} set $\G_i\subset M_\pi$, there is no \emph{a priori} guarantee that the sheaf defined in \eqref{eqn:sheaf} is invertible.
	However, by starting with a globally generated line bundle and taking $\G_i$ to be the characters corresponding to its global generators, we see \emph{a postiori} that the sheaves we describe in Proposition \ref{prop:bundle} below are indeed line bundles on $Z_\pi$.
\end{rem}

Fix a $\pi$-face $\sigma=\{v_0,\ldots,v_\ell\}$ of $\tau$, with $\pi(v_i)=e_i$.
Set	\begin{equation}\label{eqn:Gi}
		\G_i= \pi^{-1}(e_i)-v_i \subset M_\pi.
	\end{equation}
\begin{prop}\label{prop:bundle}
The universal bundle $\E$ of $\ell+1$-planes in $\KK^{\#\A}$ whose projectivization is contained in $Y_\A$ is the direct sum of line bundles
	\begin{equation}\label{eqn:E}
		\E=\bigoplus_{i=0}^\ell \mcL_i^{*},
	\end{equation}
	where each $\mcL_i$ is the line bundle on $Y_{\A_\pi}$ globally generated by the sections $\{\chi^w\ |\ w\in \G_i\}$
	for  $\G_i$ as defined in \eqref{eqn:Gi}.
\end{prop}
\begin{proof}
Let $n+1=\#\A$.
On the chart $U_\sigma$ of $Z_\pi$ containing the torus fixed point $[L_\sigma]$ corresponding to $\sigma$, the family of $\ell+1$-planes in $\KK^{n+1}$  parametrized by $Z_\pi$ is given by the rowspan of
\[
\begin{blockarray}{cccccc}
& \pi(u)=e_0 &  \pi(u)=e_1 &\cdots&  \pi(u)=e_\ell & u\notin \tau \\
\begin{block}{c(c|c|c|c|c)}
   &  \chi^{u-v_0}  &0 & &0& 0  \\
   & 0&  \chi^{u-v_{1}} &   &0 & 0  \\
  & \vdots       &&\ddots &&\vdots   \\
 &  0&0&  &\chi^{u-v_{\ell}} & 0  \\
\end{block}
\end{blockarray}\ 
\]
(see \cite[\S4]{toric}).
From the structure of this matrix, it is apparent that on $U_\sigma$, the universal bundle $\E$ decomposes as a direct sum of line bundles $\mcL_i^*$, corresponding to the rows of the matrix.
By varying $\sigma$, we see that the above decomposition glues to give a global decomposition of $\E$ as a direct sum of line bundles.
Each row of the above matrix gives a subbundle of $\CO_{Z_\pi}^{n+1}$, so dually
 we obtain a surjection $\CO_{Z_\pi}^{n+1}\to \mcL_i$, and $\mcL_i$ is globally generated.

 We have now seen that $\E$ splits as a direct sum of globally generated line bundles $\mcL_i$. It remains to describe these bundles $\mcL_i$.
 Fix an index $i$.  Since for $u\in\pi^{-1}(e_i)$, $\chi^{u-v_i}$ is a rational function on $Z_\pi$, the $i$th row of the above matrix gives a rational section $s$ of $\mcL_i^*$ (which is regular on $U_\sigma$). This rational section identifies $\mcL_i^*$ (and hence $\mcL_i$) as a subsheaf of $\KK(Z_{\pi})$, the sheaf of rational functions on $Z_\pi$.
Restricting to $U_\sigma$, we have $\CO_{U_\sigma}\to (\mcL_i^*)_{|U_\sigma}\to \CO_{U_\sigma}^{n+1}$, with the first map given by the section $s$, and the second obtained by viewing $\mcL_i^*$ as a subbundle of $\CO_{U_\sigma}^{n+1}$.
Taking the dual of $\CO_{U_\sigma}\to (\mcL_i^*)_{|U_\sigma}\to \CO_{U_\sigma}^{n+1}$, we obtain $\CO_{U_{\sigma}}^{n+1}\to \CO_{U_\sigma}$. Viewing $\mcL_i$ as a subsheaf of $\KK(Z_\pi)$, the generators of $\mcL_i$ are exactly the coordinates of $s$, that is,
 $\chi^{u-v_i}$ for $u\in \pi^{-1}(e_i)$.
\end{proof}
\begin{rem}
By choosing a different $\pi$-face, we will obtain isomorphic line bundles $\mcL_i$, albeit with a different $T_\pi$-linearization.
\end{rem}

\begin{lemma}\label{lemma:va}
Assume that $Y_\A$ is nonsingular.
Then the semigroup generated by the set $\{ w -v_j \ | \ w \in \pi^{-1}(e_j)\}$ 
is independent of $j=0,\ldots,\ell$.
\end{lemma}
\begin{proof}
	We adapt arguments from \cite[Proof of Theorem 7.3]{toric}. Let $q=\dim \tau$. 
	Considering the vertex $v_0$ of $\A$ and using the smoothness of $Y_\A$, we can conclude that there are $w_1,\ldots,w_{q-\ell}\in \tau$ such that 
\[
	w_1-v_0,\ldots,w_{q-\ell}-v_0,v_1-v_0,\ldots,v_{\ell}-v_0
\]
are a basis for the semigroup generated by $\tau-v_0$.
We order the $w_i$ so that exactly $w_1,\ldots,w_p\in \pi^{-1}(e_0)$. Here $p$ is clearly the dimension of the convex hull of $\pi^{-1}(e_0)$.

Our first claim is that $p=q-\ell$. Indeed, if not, we can extend the Cayley structure $\pi$ to $\pi':\tau\to \Delta_{\ell+1}$ by sending 
\begin{equation}\label{eqn:u}
	u=v_0+\sum_{i=1}^{q-\ell} a_i(w_i-v_0)+\sum_{i=1}^\ell b_i(v_i-v_0)\in\tau
\end{equation}
to
\[
	e_0+a_{q-\ell}(e_{\ell+1}-e_0)+\sum_{i=1}^\ell b_i(e_i-e_0).
\]
This is clearly affine, and the image certainly contains $\Delta_{\ell+1}$. But by applying $\pi$ to \eqref{eqn:u}, we see that $\sum_i b_i+\sum_{i>p} a_i\leq 1$, hence the image equals exactly $\Delta_{\ell+1}$. The construction of $\pi'$ contradicts the maximality of $\pi: \tau \to \Delta_\ell$, so henceforth we assume that all $w_i\in \pi^{-1}(e_0)$.

We now consider $u\in \pi^{-1}(e_j)$. Then by writing $u$ uniquely as in Equation~\eqref{eqn:u} and applying $\pi$ again in this case, we see that $b_i = \delta_{ij}$. We can then conclude that every element of $\pi^{-1}(e_j)-v_j$ is in the semigroup generated by $\pi^{-1}(e_0)-v_0$.
Reversing the roles of $0$ and $j$, we obtain the claim of the lemma.
\end{proof}

\begin{thm}\label{thm:chern}
Assume that $Y_\A$ is nonsingular and all $\delta_i\geq 0$. Let $\cS$ be the tautological subbundle on $\Gr_{Z_\pi}(k+1,\E)$, with $\E$ as in \eqref{eqn:E}. Then $V_{\pi,k}$ may be identified with the zero locus of a global section of the bundle
\[
\bigoplus_{j=1}^r\Sym^{\delta_j} \cS^*.
\]
In particular, if the hypotheses of Theorem \ref{thm:empty} or Corollary \ref{cor:empty} hold, then the class of $V_{\pi,k}$ in the Chow ring of $\Gr_{Z_\pi}(k+1,\E)$ is
\begin{equation}\label{eqn:chern}
\prod_{i=1}^r c_{k+\delta_i\choose k}(\Sym^{\delta_i}\cS^*).
\end{equation}
\end{thm}
\begin{proof}
	Since $Y_\A$ is nonsingular, so is $\bF_k(Y_\A)$ (Theorem \ref{cor:smooth}) and $Z_{\pi,k}$ is scheme-theoretically a smooth irreducible component of $\bF_k(Y_\A)$. We first claim that 
	\[
		Z_{\pi,k}\cong \Gr_{Z_\pi}(k+1,\E).
	\]
	Indeed,
	there is a natural morphism
	\[
		\Gr_{Z_\pi}(k+1,\E)\to Z_{\pi,k}
	\]
	sending $(x,[W])\in Z_{\pi}\times \Gr(k+1,\E_x)$ to $[\PP(W)]\in\bF_k(Y_\A)$. Here $\E_x$ is the fiber of the bundle $\E$ at $x$, and  $W$ is a $k+1$-dimensional linear space contained in $\E_x$.
This morphism is clearly surjective.

We claim that this morphism is also injective. Indeed, we need to show that for any $k$-plane $L$ corresponding to a point of $Z_{\pi,k}$, there is a unique $\ell$-plane $L'$ with $[L']\in Z_\pi$ and $L\subset L'$. Assume without loss of generality that $L$ is contained in the Pl\"ucker chart of the Grassmannian containing the torus fixed point corresponding to a $\pi$-face $\{v_0,\ldots,v_k\}$, with $\pi(v_i)=e_i$.  The local description of $Z_{\pi,k}$ from \cite[\S4]{toric} (in particular Equations 1 and 2 of loc.~cit.) shows that there is indeed a unique $L'$  as long as for each fixed $j$, the semigroup generated by $\{w-v_j\ |\ w\in \pi^{-1}(e_j)\}$ contains  $w-v_{j'}$ for all $j'$ and $w\in\pi^{-1}(e_{j'})$. 
But this criterion now follows from Lemma \ref{lemma:va}.

We know $Z_{\pi,k}$ is normal, since it is nonsingular by Corollary \ref{cor:smooth}.
It thus follows from 
Zariski's Main Theorem
that the map 
		$\Gr_{Z_\pi}(k+1,\E)\to Z_{\pi,k}$ is an isomorphism.
		We will henceforth identify $Z_{\pi,k}$ with $\Gr_{Z_\pi}(k+1,\E)$.

		We now adapt the argument from the proof of  \cite[Proposition 6.4]{EH}. 
		Consider some point $y=(x,[W])$ of $\Gr_{Z_\pi}(k+1,\E)$ as above.
	Fix sections $g_j\in H^0(Y_\A,\CO(D_j))$ determining $X$, where $D_j$ is a torus invariant divisor of class $\alpha_j$ whose support does not contain $Y_\tau$.	
		The fiber of $\Sym^{\delta_j} \cS^*$ at $y$ is 
		$H^0(L,\CO_L(\delta_j))$ where $L=\PP(W)$. Sending $(x,[W])$ to the image of $g_j$ in  $H^0(L,\CO_L(\delta_j))$ thus gives a section $s_j$ of $\Sym^{\delta_j} \cS^*$;
here we are using the canonical isomorphisms of Remark \ref{rem:canon}
for the maps $H^0(Y_\A,\CO(D_j))\to H^0(L,\CO_L(\delta_j))$.
Combining these sections $s_j$, we obtain the desired section $s=(s_j)$ of 
\[
\bigoplus_{j=1}^r\Sym^{\delta_j} \cS^*.
\]
	
This may be seen more explicitly by working locally. First fix a torus invariant affine chart $U'$ of $Z_\pi$ (corresponding to a $\pi$-face $\sigma=\{v_0,\ldots,v_\ell\}$) and then fix a Pl\"ucker chart 
\[U\cong U'\times\bA^{(k+1)(\ell-k)}\] of the Grassmannian over $U'$ (corresponding to selecting $k+1$ of the $v_i$).
	We assume without loss of generality that we are on the chart obtained by choosing $v_0,\ldots,v_k$, and $\pi(v_i)=e_i$.
	On the affine chart $U'$, the bundle $\E$ trivializes, and the bundle $\CO_U^{k+1}\cong \cS_{|U}\subset \CO_U^{\#\A}$ is given by the rows of the matrix
\begin{equation*}\label{eqn:rowspan}
B=\begin{blockarray}{cccccc}
& \pi(u)=e_0 &  \pi(u)=e_1 &  \pi(u)=e_k & \pi(u)=e_j,\ j>k & u\notin \tau \\
\begin{block}{c(c|c|c|c|c)}
   &  \chi^{u-v_0}  &0  &0&\lambda_{0j}\chi^{u-v_j} & 0  \\
   & 0&  \chi^{u-v_{1}}    &0&\lambda_{1j}\chi^{u-v_j} & 0  \\
  & \vdots       & &&\vdots & \vdots  \\
 &  0&0&  \chi^{u-v_{k}} &\lambda_{kj}\chi^{u-v_j} & 0  \\
\end{block}
\end{blockarray}
\end{equation*}
where $\lambda_{ij}$ are coordinates for $\bA^{(k+1)(\ell-k)}$, see \cite[\S4]{toric}. 
Let $b_0,\ldots,b_k$ be the standard generators of $\CO_U^{k+1}$; under the isomorphism $\CO_U^{k+1}\cong \cS_{|U}$, they give generators for $\cS_{|U}$.

For each $j$, let $\cdj{v_i}\in M$ be such that in a neighborhood of the torus fixed point of $Y_\A$ corresponding to $v_i$, $D_j=\Div \chi^{-\cdj{v_i}}$ as in \S\ref{sec:restrict}.
The Laurent polynomial $g_j$ can be written as the product of $\chi^{\cdj{v_0}}$ with a polynomial in $\chi^{u-v_0}$ as $u$ ranges over the elements of $\A$. Substituting the $u$th column of 
$(b_0/b_0,b_1/b_0,\ldots,b_k/b_0)\cdot B$ in for each $\chi^{u-v_0}$ in $b_0^{\delta_j}\cdot g_j\cdot \chi^{-\cdj{v_0}}$ leads to a polynomial $H_j(b_0,\ldots,b_k)$ of degree $\delta_j$. 
This is the same as substituting the $u$th column of 
$(b_0/b_i,b_1/b_i,\ldots,b_k/b_i)\cdot B$ in for $\chi^{u-v_i}$ in $b_i^{\delta_j}\cdot g_j\cdot \chi^{-\cdj{v_i}}$.

We now come to the explicit description of $s_j$. Viewed dually, $s_j$ is locally given as a map $\Sym^{\delta_j}\CO_U^{k+1}\to \CO_U$.
This map is obtained by sending any degree $\delta_j$ monomial in the $b_i$ to its coefficient in $H_j(b_0,\ldots,b_k)$.

Via this description, we see that the zero locus of $s$ is exactly the scheme $V_{\pi,k}$.
Indeed, working locally on $Y_\A$ near the fixed point corresponding to $v_0$, the divisor $D_j+\Div g_j$ is locally cut out by $f=\chi^{-\cdj{v_0}}\cdot g_j$. 
The condition that the above substitution of $b_0^{\delta_j}g_j\chi^{-\cdj{v_0}}$ vanishes is exactly the condition that $f$ vanishes on a given linear space.

 Finally, if the hypotheses of Theorem \ref{thm:empty} or Corollary \ref{cor:empty} hold, then $V_{\pi,k}$ has the expected dimension, so its codimension in $Z_{\pi,k}$ is just the rank of $\bigoplus_{j=1}^r \Sym^{\delta_j} \cS^*$. Hence, the zero locus of the section $s$ has the correct codimension, and its class in the Chow ring is given by the top Chern class of the bundle. This, in turn, is a product of the stated Chern classes by the Whitney sum formula.
\end{proof}
\begin{rem}
In the case that Corollary \ref{cor:xsmooth} holds and the expected dimension of $V_{\pi,k}$ is zero, 
the number of $k$-planes contained in $X$ of the type parametrized by $Z_{\pi,k}$ is exactly equal to 
the degree of \eqref{eqn:chern}.
\end{rem}

\begin{rem}
Assume that $Y_\A$ is nonsingular.
The Chow ring of $\Gr_{Z_\pi}(k+1,\E)$, and the class of $V_{\pi,k}$, can be understood quite explicitly.
Indeed, $Z_\pi$ is a nonsingular toric variety; its Chow ring is an explicit Stanley-Reisner ring, see e.g.~\cite[\S12.5]{cls}. An explicit description of the Chow ring of $\Gr_{Z_\pi}(k+1,\E)$ follows from \cite{sgs}. To compute the class of $V_{\pi,k}$, one  uses the splitting principal and Whitney's formula. All these computations may be carried out using the \texttt{Macaulay2} \cite{M2} packages  \texttt{Schubert2} \cite{schubert2} coupled with \texttt{NormalToricVarieties} \cite{tv}.
\end{rem}

\begin{figure}
	\begin{tikzpicture}
\draw (0,0)--(2,0)--(0,2)--(0,0);
\draw[fill] (0,0) circle[radius=.07cm];
\draw[fill] (1,0) circle[radius=.07cm];
\draw[fill] (2,0) circle[radius=.07cm];
\draw[fill] (1,1) circle[radius=.07cm];
\draw[fill] (0,2) circle[radius=.07cm];
\draw[fill] (0,1) circle[radius=.07cm];
\end{tikzpicture}
\hspace{1cm}
	\begin{tikzpicture}
\draw (0,0)--(1,0)--(0,1)--(0,0);
\draw[fill] (0,0) circle[radius=.07cm];
\draw[fill] (1,0) circle[radius=.07cm];
\draw[fill] (0,1) circle[radius=.07cm];
\end{tikzpicture}
\hspace{1cm}
	\begin{tikzpicture}
\draw (0,0)--(1,0)--(0,1)--(0,0);
\draw[fill] (0,0) circle[radius=.07cm];
\draw[fill] (1,0) circle[radius=.07cm];
\draw[fill] (0,1) circle[radius=.07cm];
\end{tikzpicture}
\hspace{1cm}
	\begin{tikzpicture}
\draw (0,0)--(1,0)--(0,1)--(0,0);
\draw[fill] (0,0) circle[radius=.07cm];
\draw[fill] (1,0) circle[radius=.07cm];
\draw[fill] (0,1) circle[radius=.07cm];
\end{tikzpicture}

	\caption{Fibers of $\pi_1$ for Example~\ref{ex:5}}\label{fig:fibers}
\end{figure}
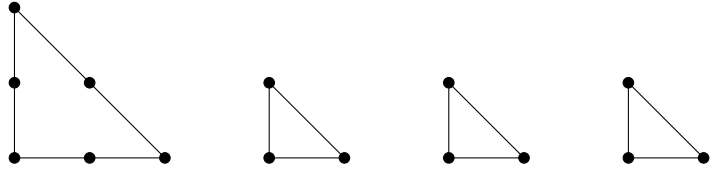

\begin{ex}[$\Bl_{\PP^2}\PP^5$]\label{ex:5}
	We conclude the example of $\Bl_{\PP^2}\PP^5$ from Examples \ref{ex:1}, \ref{ex:2}, \ref{ex:3}, and \ref{ex:4}.
	We first consider the Cayley structure $\pi_1$. Then $Z_{\pi_1,3}=\PP^2$ with universal bundle $\E_1=\CO(-2)\oplus\CO(-1)\oplus\CO(-1)\oplus\CO(-1)$.
	Indeed, the fibers of $\pi_1$ consist of the columns of the following four matrices:
\[
\left({\begin{array}{ccccccccccccccc}
0&0&2&0&1&1\\       
0&2&0&1&0&1\\
2&0&0&1&1&0\\
0&0&0&0&0&0\\
0&0&0&0&0&0
\end{array}}\right),\quad
\left({\begin{array}{ccccccccccccccc}
0&0&1\\
0&1&0\\
1&0&0\\
0&0&0\\
0&0&0
\end{array}}\right),\quad
\left({\begin{array}{ccccccccccccccc}
0&0&1\\
0&1&0\\
1&0&0\\
1&1&1\\
0&0&0
\end{array}}\right),\quad
\left({\begin{array}{ccccccccccccccc}
0&0&1\\
0&1&0\\
1&0&0\\
0&0&0\\
1&1&1
\end{array}}\right).
\]
After choosing $v_i$ from the first columns of the above matrices and taking $(1,0,-1,0,0)$ and $(0,1,-1,0,0)$ as a basis of $M_\pi$, we picture $\pi^{-1}(e_i)-v_i$ in Figure \ref{fig:fibers}.

		For 
		$\cS_1$ the tautological subbundle on $\Gr(2,\E_1)$, one computes using \texttt{Schubert2}  that 
\[
	\int c_{6}(\Sym^5 \cS_1^*)=77875.
\]
\vspace{.5cm}
\begin{verbatim}
i1 : loadPackage "Schubert2";
i2 : Z=projectiveBundle 2;
i3 : G=flagBundle({2},OO_Z(-1)+OO_Z(-1)+OO_Z(-1)+OO_Z(-2));
i4 : B=symmetricPower(5,dual (G.SubBundles)_1);
i5 : integral chern(rank B,B)

o5 = 77875
\end{verbatim}
\vspace{.5cm}
Thus, for sufficiently general $X\subset \Bl_{\PP^2}\PP^5$, $V_{\pi_1,1}$ consists of $77875$ isolated points.

	We next consider the Cayley structure $\pi_2$. Then $Z_{\pi_2,2}=\PP^2$ with universal bundle $\E_2=\CO(-1)\oplus\CO(-1)\oplus\CO(-1)$. 
	For 
$\cS_2$ the tautological subbundle on $\Gr(2,\E_2)$, one similarly computes that 
\[
	\int c_{4}(\Sym^3 \cS_2^*)=189.
\]
\vspace{.5cm}
\begin{verbatim}
i1 : loadPackage "Schubert2";
i2 : Z=projectiveBundle 2; 
i3 : G=flagBundle({1}, OO_Z(-1)+OO_Z(-1)+OO_Z(-1));
i4 : B=symmetricPower(3, dual (G. SubBundles)_1);
i5 : integral chern(rank B,B)

o5 = 189
\end{verbatim}
\vspace{.5cm}
Thus, for sufficiently general $X\subset \Bl_{\PP^2}\PP^5$, $V_{\pi_2,1}$ consists of $189$ isolated points.
We conclude that sufficiently general $X$ contains precisely $78064=77875+189$ lines.
\end{ex}

We conclude the paper with one final straightforward yet important example.
\begin{ex}[Products of projective space]\label{ex:final}
	We consider \[Y_\A=\PP^{m_1}\times\cdots\times \PP^{m_q}\] in its Segre embedding. The corresponding set $\A$ consists of those
	\[(u_1,\ldots,u_q)\in \ZZ^{m_1}\oplus \cdots\oplus \ZZ^{m_q}=M\]
	with each $u_i=(u_{i1},\ldots,u_{im_i})$ satisfying $u_{ij}\geq 0$ for all $j$ and $\sum_j u_{ij}\leq 1$. There are exactly $q$ maximal Cayley structures $\pi_i:\A\to \Delta_{\ell_i}$, with $\ell_i=m_i$ and $\pi_i$ the projection to the $i$th factor $\ZZ^{m_i}$ of $M$ (coupled with an identification of the image with $\Delta_{\ell_i}$).

	We consider a complete intersection $X$ of type $\balpha=(\alpha_1,\ldots,\alpha_r)$. Each $\alpha_i\in \Pic(Y_\A)$ is a multidegree 
	\[
		\alpha_i=(\alpha_{i1},\ldots,\alpha_{iq})\in\ZZ^q
	\]
so we may view $\balpha$ as an $r\times q$ matrix. For $\alpha_i$ to be effective, all entries must be non-negative, in which case it is already globally generated. We henceforth assume this is the case.
For the Cayley structure $\pi_j$, the restriction degree of $\alpha_i$ is $\delta_i=\alpha_{ij}$.

For any $k\geq 1$ and Cayley structure $\pi_j$, the expected dimension is 
\[
\phi=	\phi(\A,\pi_j,\balpha,k)=\sum_{i\neq j} m_i+(k+1)(m_j-1)-\sum_{i=1}^r {k+\alpha_{ij} \choose k}.
\]
As long as $V_{\pi_j,k}$ is non-empty, this gives a lower bound on its dimension. For $V_{\pi_j,k}$ non-empty and $X$ general, if $\phi\geq 0$ then $\dim V_{\pi_j,k}=\phi$ (Theorem \ref{thm:expected}).

Assume that all entries of $\balpha$ are greater than zero, $\phi\geq 0$, and $r+2k\leq \sum_i m_i$. Assume further that either some $\alpha_{ij}\geq 3$, there are at least two $\alpha_{ij}$ which are larger than one, or $r+2k\leq m_j$. Then Corollaries \ref{cor:empty} and \ref{cor:xsmooth} apply, and we conclude that $V_{\pi_j,k}$ is non-empty of the expected dimension, and is smooth if $X$ is generic.

For the Cayley structure $\pi_j$, $Z_{\pi_j}$ is just 
\[
	Z_{\pi_j}=\prod_{i\neq j} \PP^{m_i}
\]
and the universal bundle $\E$ (Proposition \ref{prop:bundle}) is  \[\E=\bigoplus_{i=0}^{m_j}\CO_{Z_{\pi_j}}(-1,\ldots,-1).\]
Under our above assumptions, the class of $V_{\pi_j,k}$ in the Grassmann bundle $\Gr(k+1,\E)$ over $Z_{\pi,j}$ is
\begin{equation*}
	\prod_{i=1}^r c_{k+\alpha_{ij}\choose k}(\Sym^{\alpha_{ij}}\cS^*)
\end{equation*}
by Theorem \ref{thm:chern}.

Specializing to $q=2$, $m_1=m_2=2$, $r=1$, $\alpha_1=(3,3)$, and $k=1$, we obtain Example \ref{ex:first} from the introduction and all of the assumptions in this example hold. 
The degree computation of $189$ is exactly the same as the second degree computation in Example \ref{ex:5}.
\end{ex}

\bibliographystyle{alpha}
\bibliography{hypersurfaces} 
\end{document}